\documentclass[a4paper]{amsart}
\usepackage{amssymb,enumitem,nicefrac}
\usepackage{mathtools}
\usepackage[all]{xy}
\usepackage{hyperref}
\usepackage[lite]{amsrefs}%for references
\usepackage{microtype}%microtypography for pdflatex avoids most bad boxes

\newtheorem{theorem}{Theorem}[section]
\newtheorem{lemma}[theorem]{Lemma}
\newtheorem{proposition}[theorem]{Proposition}
\theoremstyle{remark}
\newtheorem{example}[theorem]{Example}
\newtheorem{remark}[theorem]{Remark}
\theoremstyle{definition}
\newtheorem{definition}[theorem]{Definition}

\newcommand*{\abs}[1]{\lvert#1\rvert}%absolute value
\newcommand*{\norm}[1]{\left\| #1\right\|}%norm
\newcommand*{\gen}[1]{\langle#1\rangle}%subgroup generated by
\newcommand*{\oppa}[1]{\bar#1}%opposite parabolic subgroup
\newcommand*{\cl}[1]{\overline{#1}}%closure

\newcommand*{\Ccinf}{\textup C^\infty_\textup c}%locally constant function with compact support

\newcommand*{\Oest}{\textup{O}}%O-notation for estimates
\newcommand*{\id}{\textup{id}}%identity map
\newcommand*{\reg}{\textup{rss}}%regular semisimple
\newcommand*{\red}{\textup{red}}%reduced
\newcommand*{\GL}{\textup{GL}}%general linear group
\newcommand*{\diff}{\textup d}%differential as in dx
\newcommand*{\nb}{\nobreakdash}%no word break after the following hyphen
\newcommand*{\defeq}{\mathrel{\vcentcolon=}}%used for definitions
\newcommand*{\eqdef}{\mathrel{=\vcentcolon}}%used for definitions
\DeclareMathOperator{\Endo}{End}%endomorphism ring
\DeclareMathOperator{\Aut}{Aut}%automorphism group
\DeclareMathOperator{\Ad}{Ad}%conjugation action
\DeclareMathOperator{\tr}{tr}%trace
\DeclareMathOperator{\Span}{span}%linear span
\DeclareMathOperator{\Ind}{Ind}%induction of representations
\DeclareMathOperator{\dom}{dom}%domain of a function
\DeclareMathOperator{\image}{im}%image of a function
\DeclareMathOperator{\hgt}{ht}%height
\newcommand*{\Lie}[1][\F]{\operatorname{Lie}_{#1}}%F-points of a Lie algebra

\newcommand{\C}{\mathbb C}%complex numbers
\newcommand{\R}{\mathbb R}%real numbers
\newcommand{\Q}{\mathbb Q}%rational numbers
\newcommand{\Z}{\mathbb Z}%integers
\newcommand{\N}{\mathbb N}%natural numbers
\newcommand{\F}{\mathbb F}%p-adic field
\newcommand{\K}{\mathbb K}%another field of good characteristic

\newcommand*{\Galg}{\mathcal G}%algebraic group
\newcommand*{\G}{G}%F-points of the algebraic group
\newcommand*{\Talg}{\mathcal T}%(maximal) algebraic torus
\newcommand*{\T}{T}%F-points of torus
\newcommand*{\Salg}{\mathcal S}%maximal split algebraic torus
\newcommand*{\ST}{S}%F-points of maximal split torus
\newcommand*{\Zalg}[1][\Galg]{\mathcal Z(#1)}%centre
\newcommand*{\Ze}[1][\G]{\textup Z(#1)}%F-points of the centre
\newcommand*{\Zcalg}[1][\Galg]{\mathcal Z_\textup c(#1)}%connected centre
\newcommand*{\Zc}[1][\G]{\textup Z_\textup c(#1)}%F-points of the connected centre
\newcommand*{\Ralg}[1][\Palg]{\mathcal R_\textup u(#1)}%unipotent radical
\newcommand*{\Ru}[1][\Pa]{\textup R_\textup u(#1)}%unipotent radical
\newcommand*{\Ualg}{\mathcal U}%unipotent algebraic subgroup
\newcommand*{\Un}{U}%F-points of unipotent subgroup
\newcommand*{\UC}[1][\Omega]{U_{#1}}%compact subgroup for subset of apartment
\newcommand*{\NC}[1][\Omega]{N_{#1}}%normaliser subgroup for subset of apartment
\newcommand*{\PC}[1][\Omega]{P_{#1}}%stabiliser subgroup for subset of apartment
\newcommand*{\tUC}[1][\Omega]{\tilde{U}_{#1}}%compact subgroup for subset of apartment over extension field
\newcommand*{\Calg}[1][\Galg]{\mathcal Z_{#1}}%centraliser
\newcommand*{\Ce}[1][\G]{\textup Z_{#1}}%F-points of centraliser
\newcommand*{\Nalg}[1][\Galg]{\mathcal N_{#1}}%normaliser
\newcommand*{\No}[1][\G]{\textup N_{#1}}%F-points of normaliser
\newcommand*{\Lalg}{\mathcal M}%algebraic Levi subgroup
\newcommand*{\Le}{M}%F-points of the Levi subgroup
\newcommand*{\Palg}{\mathcal P}%algebraic parabolic subgroup
\newcommand*{\Pa}{P}%F-points of the parabolic subgroup
\newcommand*{\CG}{K}%compact subgroup of \G
\newcommand*{\integ}{\mathcal O}%integers in a local field
\newcommand*{\maxid}{\mathfrak P}%maximal ideal in \(\integ\)
\newcommand*{\Zinteg}{\mathfrak Z}%integral version of Z_G(S)

\newcommand*{\Apa}[1][\ST]{A_{#1}} %Apartment corresonding to #1
\newcommand*{\tApa}[1]{\tilde{A}_{#1}} %Apartment corresponding to #1 over extension field
\newcommand*{\piTx}{\pi_\T (x)}%nearest point to x in \tilde{A}
\newcommand*{\Kx}{K_x} %specific compact subgroup

\newcommand*{\sd}{\textup{sd}}%singular depth

\newcommand*{\Phr}{\Phi^\red}%reduced root system
\newcommand*{\Hec}{\mathcal H (G,\mathbb Z [1/p])}%Hecke algebra
\newcommand*{\ide}[1]{\langle #1 \rangle}%idempotent corresponding to #1
\newcommand*{\repr}{\pi}%representation
\newcommand*{\trpi}{\tr_\pi}%trace function
\newcommand*{\Ho}{\textup H}%homology of a chain complex
\newcommand*{\Cell}[1][*]{C_{#1}}%cellular chain complex with coefficients in a cosheaf
\newcommand*{\Hull}{\mathcal H}%convex hull of a set of facets
\newcommand*{\Hecke}{\mathcal H}%Hecke algebra
\newcommand*{\eF}{e_{\tilde\F/\F}}%ramification index
\newcommand*{\rmax}{r(\gamma)}%radius of set where \trpi is constant

\newcommand{\ceil}[1]{\lceil#1{+}\rceil}
\newcommand{\inp}[2]{\langle#1,#2\rangle}
\newcommand{\bt}[2][\Galg]{\mathcal B(#1,#2)}%building of #1 over #2

\begin{document}
\title[Characters and growth of admissible representations]{Characters and growth of admissible\\
representations of reductive \(p\)-adic groups}
\author{Ralf Meyer}
\email{rameyer@uni-math.gwdg.de}
\author{Maarten Solleveld}
\email{maarten@uni-math.gwdg.de}

%address for both
\address{Mathematisches Institut and Courant Centre ``Higher order structures''\\
  Georg-August Universit\"at G\"ottingen\\
  Bunsenstra{\ss}e 3--5\\
  37073 G\"ottingen\\
  Germany}

\begin{abstract}
  We use coefficient systems on the affine Bruhat--Tits building to study admissible representations of reductive \(p\)\nb-adic groups in characteristic not equal to~\(p\).  We show that the character function is locally constant and provide explicit neighbourhoods of constancy.  We estimate the growth of the subspaces of invariants for compact open subgroups.\\
{\sc Mathematics Subject Classification} (2010). 20G25, 20E42
\end{abstract}
\thanks{Supported by the German Research Foundation (Deutsche
  Forschungsgemeinschaft (DFG)) through the Institutional Strategy of
  the University of G\"ottingen.}
\maketitle
\tableofcontents

\section{Introduction}
\label{sec:intro}

Let~\(\F\) be a non-Archimedean local field, possibly of nonzero characteristic, and let~\(\G\) be a reductive algebraic group over~\(\F\), briefly called a reductive \(p\)\nb-adic group.  Let~\(\repr\) be an admissible representation of~\(\G\) on a complex vector space~\(V\).  Since~\(V^\CG\) has finite dimension for every compact open subgroup \(\CG \subseteq \G\), the operator \(\repr(f)\) has finite rank for all test functions~\(f\).  The resulting distribution \(\theta_\repr(f) \defeq \tr(\repr (f) ,V)\) is called the \emph{character} of~\(\repr\).  Since~\(V\) usually has infinite dimension, the operators \(\repr(g)\) need not be trace-class for \(g \in\nobreak \G\).  Nevertheless, Harish-Chandra could show that the character is described by a locally integrable function:

\begin{theorem}[Harish-Chandra]
  \label{thm:HC}
  Let \(\repr\colon \G\to\Aut(V)\) be an admissible representation of a reductive \(p\)\nb-adic group.

  \begin{enumerate}[label=\textup{(\alph{*})}]
  \item\label{thm:HC_a} The operator \(\repr(g)\) has a well-defined trace \(\trpi(g)\) when~\(g\) belongs to the set \(\G_\reg\) of regular semisimple elements.
  \item\label{thm:HC_b} The function \(\trpi\colon \G_\reg \to \C\) is locally constant.
  \item\label{thm:HC_c} The function~\(\trpi\), extended by~\(0\) on \(\G\setminus \G_\reg\), is locally integrable with respect to the Haar measure~\(\mu\) on~\(\G\), and for any test function~\(f\),
    \[
    \theta_\repr(f) = \int_\G f(g) \trpi(g) \,\diff\mu(g).
    \]

  \item\label{thm:HC_d} Let \(D(g)\) for \(g\in \G_\reg\) be the determinant of \(\Ad(g) - 1\) acting on \(\Lie[](\G) \mathbin/ \Lie[](\T)\) for a maximal torus~\(\T\) in~\(\G\) containing~\(g\).  The function \(\G \ni g \mapsto |D(g)|^{1/2} \trpi (g)\) is locally bounded.
  \end{enumerate}
\end{theorem}

The original proof of this deep theorem is distributed over various papers of Harish-Chandra collected in~\cite{Harish-Chandra:Collected_IV}.  A complete account of it can be found in~\cite{Harish-Chandra:Admissible_distributions}.  The proofs of \ref{thm:HC_c} and~\ref{thm:HC_d} use the exponential mapping for~\(\G\), which only works well if the characteristic of~\(\F\) is zero.  It is reasonable to expect that \ref{thm:HC_c} and~\ref{thm:HC_d} are valid in non-zero characteristic as well, but the authors are not aware of a proof.  According to \cite{Vigneras-Waldspurger:Premiers}*{paragraph E.4.4} Harish-Chandra's proof of \ref{thm:HC_a} and \ref{thm:HC_b} remains valid if one replaces~\(\C\) by an algebraically closed field of characteristic unequal to~\(p\).

In this article we generalise part of Theorem~\ref{thm:HC} to representations on modules over unital rings in which~\(p\) is invertible.  In this purely algebraic setting, we can only define the character as a function where it is locally constant.  To prove \ref{thm:HC_a} and~\ref{thm:HC_b}, we describe explicit neighbourhoods on which~\(\trpi\) is constant.  In characteristic~\(0\), similar results are due to Adler and Korman~\cite{Adler-Korman:Local_character}.

Parts \ref{thm:HC_c} and~\ref{thm:HC_d} seem specific to real or complex representations because they involve analysis.  Unfortunately, our methods are insufficient to (re)prove them, as we discuss in the last section.

As a substitute we estimate the dimension of invariant subspaces~\(V^\CG\) for certain compact open subgroups~\(\CG\) in~\(\G\).  The authors have not found growth estimates for these dimensions in the literature.  Since~\(V^\CG\) is the range of an idempotent~\(\ide{\CG}\) in the Hecke algebra associated to~\(\CG\), we get
\[
\dim V^\CG = \frac{1}{\abs{\CG}} \int_\CG \trpi(g)\,\diff\mu(g).
\]
But the estimate in~\ref{thm:HC_d} is not strong enough to control these integrals.

Our methods are of a geometric nature and involve the affine building of~\(\G\).  Thus we will make extensive use of Bruhat--Tits theory, including some hard parts.  At the same time, we use only little representation theory.  Both of our main results use the resolutions constructed by Schneider and Stuhler
\cite{Schneider-Stuhler:Rep_sheaves}.  These resolutions are based on a family of compact open subgroups~\(\UC[x]^{(e)}\) for \(e\in\N\), indexed by vertices of the affine Bruhat--Tits building.  These generate subgroups~\(\UC[\sigma]^{(e)}\) indexed by polysimplices in the building.  The invariant subspaces~\(V^{\UC[\sigma]^{(e)}}\) in an admissible representation~\(V\) form a locally finite-dimensional coefficient system on the building.  It is shown in~\cite{Meyer-Solleveld:Resolutions} that this coefficient system is acyclic on any convex subcomplex of the building.  In particular, it provides a resolution of~\(V\) of finite type.

Here we need acyclicity also for finite subcomplexes of the building because this provides chain complexes of finite-dimensional vector spaces, which are used in~\cite{Meyer-Solleveld:Resolutions} to express the character of~\(V\) as a sum over contributions of polysimplices in the building.  We use this formula to find for each regular semisimple element~\(\gamma\) and each vertex~\(x\) in the building a number~\(r\) such that the character is constant on~\(\UC[x]^{(r)}\gamma\); the constant~\(r\) depends on the distance between~\(x\) and a subset of the building corresponding to the maximal torus containing~\(\gamma\), on the (ir)regularity of~\(\gamma\), and on the level of the representation~\(V\), that is, on the smallest \(e\in\N\) such that~\(V\) is generated by the \(\UC[y]^{(e)}\)\nb-invariants for all vertices~\(y\).

Along the way, we also prove some auxiliary results that may be useful in other contexts.  
We prove that the parabolic subgroup contracted by an element of a reductive \(p\)\nb-adic group is 
indeed parabolic and, in particular, algebraic (Proposition~\ref{prop:Pg}).
We describe which points in the building are fixed by a semisimple element in Section~\ref{sec:fixedpoints}.  We establish that the level of representations is preserved by Jacquet induction and restriction
(Proposition~\ref{prop:levele}).  The relationship between character function and distribution is made 
precise in an algebraic setting in Section~\ref{sec:admissible}.

\section{The structure of reductive algebraic groups}
\label{sec:reductive}

We fix our notation and recall some general facts from the theory of linear algebraic groups.  Nothing in this section is new and most of it can be found in several textbooks, for example~\cite{Springer:Linear_algebraic_groups}.

Let~\(\Galg\) be a linear algebraic group defined over a field~\(\F\).  The collections of characters and cocharacters of~\(\Galg\) are denoted by \(X^* (\Galg)\) and \(X_* (\Galg)\), respectively.  Let \(\G \defeq \Galg (\F)\) be its group of \(\F\)\nb-rational points.  By definition, an algebraic (co)character of~\(\G\) is a (co)character of~\(\Galg\) that is defined over~\(\F\).  The corresponding sets are denoted by \(X^* (\G)\) and \(X_* (\G)\).  Let~\(\Zalg\) be the centre of~\(\Galg\) and let~\(\Zcalg\) be the maximal connected algebraic subgroup of~\(\Zalg\).  We denote the centraliser in~\(\Galg\) of an element \(g \in \G\) by \(\Calg(g)\).

We will assume throughout that \(\Galg\) is connected and reductive.  An algebraic subgroup~\(\Palg\) of~\(\Galg\) is \emph{parabolic} if~\(\Galg/\Palg\) is a complete algebraic variety.  We denote the unipotent radical of~\(\Palg\) by~\(\Ralg\).  A \emph{Levi factor} of~\(\Palg\) is a reductive subgroup~\(\Lalg\) such that \(\Palg = \Lalg \ltimes \Ralg\).

We write \(\Ze\), \(\Zc\), \(\Pa\), \(\Ru\), and~\(\Le\) for the groups of \(\F\)\nb-points of \(\Zalg\), \(\Zcalg\), \(\Palg\), \(\Ralg\), and \(\Lalg\), respectively.  We denote the space of \(\F\)\nb-points of the Lie algebra of~\(\Galg\) by \(\Lie(\Galg)\).

We say that an algebraic torus~\(\Talg\) \emph{splits over~\(\F\)} if \(\Talg(\F) \cong (\F^\times )^{\dim \Talg}\) as \(\F\)\nb-groups.  We say that~\(\G\) \emph{splits} (over~\(\F\)) if there is a maximal torus~\(\Talg\) of~\(\Galg\) that splits over~\(\F\).

\begin{proposition}
  \label{prop:splitG}
  There is a finite Galois extension of~\(\F\) over which~\(\Galg\) splits.
\end{proposition}

\begin{proof}
  For tori this was first proven by Ono \cite{Ono:Arithmetic}*{Proposition 1.2.1}.  This implies the result for general reductive groups.
\end{proof}

Let~\(\Salg\) be maximal among the tori in~\(\Galg\) that split over~\(\F\) and let \(\ST \defeq \Salg (\F)\).  We call~\(\ST\) a \emph{maximal split torus} in~\(\G\).  Notice that every algebraic (co)character of~\(\Salg \) is defined over~\(\F\), as~\(\ST\) is split.  Let \(\Phi= \Phi (\Galg, \Salg) \subset X^*(\Salg)\) be the root system of~\(\Galg\) with respect to~\(\Salg\), and let \(\Phi^\vee \subset X_*(\Salg)\) be the dual root system.  Let \(\Calg(\Salg)\) and \(\Nalg(\Salg)\) denote the centraliser and the normaliser of~\(\Salg\) in~\(\Galg\) and let \(\Ce(\ST)\) and \(\No(\ST)\) be their groups of \(\F\)\nb-points.  The \emph{Weyl group} of~\(\Phi\) is
\[
W(\Phi) \defeq \No(\ST)  \bigm/ \Ce(\ST).
\]
The root system~\(\Phi\) need not be reduced if~\(\G\) is not split.  The corresponding reduced root system is
\begin{equation}
  \Phr \defeq \{\alpha \in \Phi(\Galg,\Salg) :
  \alpha/2 \notin \Phi(\Galg,\Salg)\}.
\end{equation}
For every root \(\alpha \in \Phi(\Galg,\Salg)\) there is a unipotent algebraic subgroup \(\Ualg_\alpha \subset \Galg\) with group of \(\F\)\nb-points~\(\Un_\alpha\), characterised by the following two conditions:
\begin{itemize}
\item \(\Calg(\Salg)\) normalises~\(\Ualg_\alpha\),
\item \(\Lie(\Ualg_\alpha)\) is the sum of the \(\ST\)\nb-weight spaces for \(\alpha\) and~\(2 \alpha\), with respect to the adjoint action of~\(\ST\) on \(\Lie(\Galg)\).
\end{itemize}
If \(\alpha, 2 \alpha \in \Phi\) then \(\Un_{2 \alpha} \subsetneq \Un_\alpha\), and it is convenient to write \(\Un_{2 \alpha} = \{1\}\) if \(\alpha \in \Phi\) but \(2 \alpha \notin \Phi\).  The groups \(\Un_\alpha / \Un_{2 \alpha}\) and~\(\Un_{2 \alpha}\) are naturally endowed with the structure of an \(\F\)\nb-vector space and are isomorphic to their respective Lie algebras.  The subset \(\bigcup_{\alpha \in \Phr} \Un_\alpha \cup \Ce(\ST)\) generates the group~\(\G\).

Let~\(\Phi^+\) be a system of positive roots in~\(\Phi\) and let \(\Delta \subseteq \Phr\) be the corresponding basis.  Any subset \(D \subseteq \Delta\) is a basis of a root system \(\Phi_D \defeq \Z D \cap \Phi\).  The algebraic subgroup~\(\Palg_D\) of~\(\Galg\) generated by \(\Calg(\Salg)\) and the~\(\Ualg_\alpha\) with \(\alpha \in \Phi_D \cup \Phi^+\) is parabolic.  Its unipotent radical is generated by the~\(\Ualg_\alpha\) with \(\alpha \in \Phi^+ \setminus \Phi_D^+\).  The group~\(\Lalg_D\) that is generated by \(\bigcup_{\alpha \in \Phi_D} \Ualg_\alpha \cup \Calg(\Salg)\) is a Levi subgroup of~\(\Palg_D\).  Moreover, \(\Lalg_D = \Calg(\Salg_D)\), where~\(\Salg_D\) is the connected component of
\[
\bigcap_{\alpha \in \Phi_D} \ker \alpha
\subseteq \Zalg[\Lalg_D].
\]
We note that~\(\Palg_\emptyset\) is a Borel subgroup of~\(\Galg\), that \(\Palg_\Delta = \Lalg_\Delta = \Galg\), and that~\(\Salg_\Delta (\F)\) is the unique maximal split torus of~\(\Zalg\).

\begin{definition}
  \label{def:standard_parabolic}
  Groups of the form~\(\Palg_D\) are called \emph{standard parabolic} (with respect to \(\Salg\) and~\(\Phi^+\)).
\end{definition}

Every parabolic subgroup of~\(\Galg\) is conjugate to exactly one standard parabolic subgroup.  Let \(\Phi^- \defeq - \Phi^+\) be the set of negative roots and let \(\oppa{\Palg_D}\) be subgroup of~\(\Galg\) generated by \(\Calg(\Salg)\) and the~\(\Ualg_\alpha\) with \(\alpha\in\Phi_D \cup \Phi^-\).  The parabolic subgroup \(\oppa{\Palg_D}\) is \emph{opposite} to~\(\Palg_D\) in the sense that \(\Palg_D \cap \oppa{\Palg_D} = \Lalg_D\) is a Levi subgroup of both.  Moreover
\[
\Lie (\Galg) = \Lie \bigl(\Ralg[\Palg_D]\bigr)
\oplus \Lie (\Lalg_D) \oplus \Lie \bigl(\Ralg[\bar{\Palg}_D]\bigr).
\]
We shall also need the \emph{pseudo-parabolic} subgroup
\begin{equation}
  \label{eq:pseudop}
  \Pa (\chi) \defeq \{ p \in \G : \lim_{\lambda \to 0} \chi(\lambda) p \chi(\lambda)^{-1} \text{ exists} \}
\end{equation}
for an algebraic cocharacter \(\chi\colon \F^\times\to\G\).  This limit is meant purely algebraically, by definition it exists if and only if the corresponding map \(\F^\times \to \G\) extends to an algebraic morphism \(\F \to \G\).  In a reductive group, any pseudo-parabolic subgroup is the group of \(\F\)\nb-points of a parabolic subgroup by \cite{Springer:Linear_algebraic_groups}*{Lemma 15.1.2}.

From now on we assume that the field~\(\F\) is endowed with a non-trival discrete valuation \(v\colon \F \to \Q \cup \{\infty\}\).  We fix a real number \(q > 1\) and we define a metric on~\(\F\) by
\[
d(\lambda , \mu) = q^{-v (\lambda - \mu)}.
\]
Via an embedding \(\Galg \to \GL_n \), the metric~\(d\) yields a metric on \(\G = \Galg(\F)\) as well.  Even though there is no unique way to do this, the resulting collection of bounded subsets of~\(\G\) is canonical.  This bornology on~\(\G\) is compatible with the group structure, in the sense that \(B_1^{-1} B_2\) is bounded for all bounded subsets \(B_1\) and~\(B_2\) of~\(\G\).

It follows directly from the properties of a valuation that every finitely generated subgroup of \((\F ,+)\) is bounded, and this implies that every unipotent element of~\(\G\) generates a bounded subgroup.

Following Deligne~\cite{Deligne:Support}, we assign to any \(g \in \G\) the \emph{parabolic subgroup contracted by~\(g\)},
\begin{equation}
  \label{eq:Pg}
  \Pa_g \defeq \bigl\{p \in \G :
  \text{\(\{ g^n p g^{-n} : n \in \N \}\) is bounded} \bigr\},
\end{equation}
and
\begin{equation}
  \label{eq:Mg}
  \Le_g \defeq \Pa_g \cap \Pa_{g^{-1}}
  =  \bigl\{ p \in \G : \text{\(\{ g^n p g^{-n} : n \in \Z\}\) is bounded}\bigr\}.
\end{equation}

The following result, which will be needed in Section~\ref{sec:constancygen}, was proved in \cite{Prasad:Elementary_Proof}*{Lemma 2} under the additional assumptions that~\(\Galg\) is semisimple and almost \(\F\)\nb-simple.  Although it is apparently well-known that it holds for general reductive groups, the authors have not found a good reference for this.

\begin{proposition}
  \label{prop:Pg}
  The subgroups \(\Pa_g\) and~\(\Le_g\) for \(g\in \G\) have the following properties:
  \begin{enumerate}[label=\textup{(\alph{*})}]
  \item\label{prop:Pga} \(\Pa_g\) is a parabolic subgroup of \(\G\).
  \item\label{prop:Pgb} \(\Ru[\Pa_g] = \{ p \in \G : \lim_{n\to\infty} g^npg^{-n}=1\}\).
  \item\label{prop:Pgc} The parabolic subgroup \(\Pa_{g^{-1}}\) is opposite to~\(\Pa_g\) and~\(\Le_g\) is a Levi subgroup of~\(\Pa_g\).
  \item\label{prop:Pgd} \(g \Ze[\Le_g]\) is contained in a bounded subgroup of \(\Le_g\mathbin/\Ze[\Le_g]\).
  \end{enumerate}
\end{proposition}

\begin{proof}
  We first establish~(a).  Clearly, \(\Pa_g\) is a subgroup of~\(\G\) that contains~\(g\).  The difficulty is to show that~\(\Pa_g\) is an algebraic subgroup of~\(\G\), although it is defined in topological terms.  Choose a finite extension field~\(\F_g\) of~\(\F\) which contains the roots of the characteristic polynomial of~\(g\).  Then we have a Jordan decomposition \(g = g_s g_u = g_u g_s\) in \(\Galg(\F_g )\), see \cite{Springer:Linear_algebraic_groups}*{Section 2.4}.  Let~\(\Talg\) be a maximal torus in~\(\Galg\) defined over~\(\F_g\) that contains~\(g_s\), and let~\(\tilde{\F}\) be a finite extension field of~\(\F_g\) over which~\(\Talg\) splits (Proposition~\ref{prop:splitG}).  We may and will assume that~\(\tilde{\F}\) is normal over~\(\F\).  According to \cite{Serre:Local_fields}*{Section I.4} the valuation~\(v\) extends to a valuation~\(\tilde v\) on~\(\tilde{\F}\).  We abbreviate \(\Galg(\tilde{\F}) = \tilde{\G}\), and similarly for its algebraic subgroups.  Let~\(\tilde{\Phi}\) be the root system of~\(\Galg\) with respect to~\(\Talg\).

  Since~\(g_u\) is unipotent, \(\tilde \CG \defeq \{g_u^n : n \in \Z\}\) is a bounded subgroup of~\(\tilde \G\), and it centralises~\(g_s\).  For \(\alpha \in \tilde{\Phi}\) and \(p \in \tilde \Un_\alpha \setminus \{1\}\), the following are equivalent:
  \begin{itemize}
  \item \(\{ g^n p g^{-n} : n \in \N \}\) is bounded,
  \item \(\tilde \CG \{ g_s^n p g_s^{-n} : n \in \N \} \tilde \CG\) is bounded,
  \item \(\{ g_s^n p g_s^{-n} : n \in \N \}\) is bounded,
  \item \(g_s p g_s^{-1} = \lambda p\) with \(\{ \lambda^n : n \in \N \} \subseteq \tilde{\F}\) bounded,
  \item \(\tilde{v} \bigl(\alpha (g_s)\bigr) \ge 0\).
  \end{itemize}
  We may choose a system of positive roots~\(\tilde{\Phi}^+\) with \(\tilde{v} \bigl(\alpha (g_s)\bigr) \ge 0\) for all \(\alpha\in\tilde{\Phi}^+\).  Let \(D\subseteq\tilde{\Delta}\) be the set of simple roots with \(\tilde{v}\bigl(\alpha(g_s)\bigr)=0\).  The group~\(\tilde{\Pa}_g\) is generated by \(\tilde{\T}\defeq \Talg(\tilde{\F})\) and all~\(\tilde{\Un}_\alpha\) with \(\alpha\in\tilde{\Phi}^+\cup \Phi_D\).  Thus~\(\tilde{\Pa}_g\) is the group of \(\tilde{\F}\)\nb-points of the parabolic subgroup~\(\Palg_D\) of~\(\Galg\), and the collection of non-zero weights of~\(\tilde{\T}\) in \(\Lie[\tilde{\F}] (\Palg_D)\) equals
  \begin{equation}
    \label{eq:PhiPg}
    \bigl\{\alpha\in\tilde{\Phi} :
    \tilde{v} \bigl(\alpha(g_s)\bigr) \ge 0\bigr\}
    \eqdef \Phi(\Palg_g,\Talg).
  \end{equation}

  As mentioned above, \(\tilde{\Pa}_g\) is also a pseudo-parabolic subgroup of~\(\tilde{\G}\), so there is a cocharacter \(\tilde{\chi} \in X_*(\Galg)\) with \(\tilde{\Pa}_g = \tilde{\Pa}(\tilde{\chi})\).  In fact, any \(\tilde{\chi} \in X_*(\tilde{\Talg})\) with
  \begin{equation}
    \label{eq:2.2}
    \{ \alpha \in \tilde{\Phi} :
    \inp{\alpha}{\tilde \chi} \ge 0 \}
    = \Phi(\Palg_g ,\Talg)
  \end{equation}
  will do.  To prove that \(\Pa_g = \tilde{\Pa}_g \cap \G\) is a parabolic subgroup of~\(\G\), we must find a cocharacter~\(\chi\) that satisfies~\eqref{eq:2.2} and is defined over~\(\F\).  Then \(\Pa_g = \Pa(\chi)\) will be pseudo-parabolic and hence parabolic.

  Let~\(\Gamma\) be the group of field automorphisms of~\(\tilde{\F}\) over~\(\F\).  Since \(g\in \Galg(\F)\) and~\(\Gamma\) acts continuously, the subgroup~\(\tilde{\Pa}_g\) is \(\Gamma\)\nb-invariant by~\eqref{eq:Pg}, so that \(\gamma \circ \tilde{\chi} \circ \gamma^{-1}\) satisfies~\eqref{eq:2.2} for all \(\gamma \in \Gamma\).  Since the set of solutions of~\eqref{eq:2.2} forms a cone in the free abelian group \(X_*(\tilde{\Talg})\), it contains
  \[
  \tilde{\chi}^\Gamma : \lambda \mapsto \prod_{\gamma \in \Gamma}
  \gamma \bigl( \tilde{\chi} ( \gamma^{-1} \lambda ) \bigr).
  \]
  Thus \(\tilde{\Pa}_g = \tilde{\Pa}(\tilde{\chi}^\Gamma)\).  The cocharacter~\(\tilde{\chi}^\Gamma\) is defined over~\(\tilde{\F}^\Gamma\).  The field extension \(\F \subseteq \tilde{\F}^\Gamma\) is finite and purely inseparable, see for example \cite{Lang:Algebra}*{Section 7.7}.  Hence some positive multiple~\(\chi\) of \(\tilde{\chi}^\Gamma\) is defined over~\(\F\) and still satisfies~\eqref{eq:2.2}.  This yields \(\tilde{\Pa}_g = \tilde{\Pa}(\chi)\) and finishes the proof of~(a).

  Now we prove~(b).  \(\Lie[\tilde{\F}](\Palg_g)\) is spanned by the vectors \(X \in \Lie[\tilde{\F}](\Galg)\) with \(\Ad(g_s) X = \lambda X\) with \(\tilde{v}(\lambda) \ge 0\).  Similarly, \(\Lie[\tilde{\F}]\bigl(\Ralg[\Palg_g]\bigr)\) is spanned by the root subspaces \(\Lie[\tilde{\F}](\Ualg_\alpha)\) with \(\alpha \in \Phi(\Palg_g,\Talg)\) but \(-\alpha \notin \Phi(\Palg_g,\Talg)\).  These are precisely the \(\alpha\in\Phi\) with \(\tilde{v} \bigl(\alpha(g_s)\bigr) > 0\).  Therefore
  \[
  \lim_{n \to \infty} g_s^n h g_s^{-n} = 1
  \quad\iff\quad h \in \Ru[\tilde{\Pa}_g].
  \]
  Since all powers of~\(g_u\) are contained in the bounded subgroup~\(\tilde{\CG}\), these statements are also equivalent to \(\lim_{n \to \infty} g^n h g^{-n} = 1\).  Now~(b) follows because \(\Ru[\Pa_g] = \Ru[\tilde\Pa_g] \cap \Pa_g\).

  Next we establish~(c).  Let~\(\chi\) be a cocharacter of~\(\Galg\) defined over~\(\F\) with \(\Pa_g = \Pa(\chi)\).  The same reasoning as in the proof of~(a) shows that \(\Pa_{g^{-1}} = \Pa(-\chi)\).  The assertion~\ref{prop:Pgc} now follows by applying \cite{Springer:Linear_algebraic_groups}*{Theorem 13.4.2} to \(\Pa_g\) and~\(\Pa_{g^{-1}}\).

  Finally, we turn to~(d).  The eigenvalues of \(\Ad(g_s)\) acting on \(\Lie[\tilde{\F}](\Lalg_g)\) all have valuation~\(0\).  Hence \(\Ad(g)\) lies in a bounded subgroup of the adjoint group of~\(\tilde{\Le}_g\).  Equivalently, the image of~\(g\) in \(\tilde{\Le}_g \mathbin/ \Ze[\tilde{\Le}_g]\) generates a bounded subgroup.  Finally, we note that \(\Le_g \mathbin/ \Ze[\Le_g]\) can be identified with a subgroup of \(\tilde{\Le}_g \mathbin/ \Ze[\tilde{\Le}_g]\).
\end{proof}

\section{Some Bruhat--Tits theory}
\label{sec:BruhatTits}

We keep the notation from Section~\ref{sec:reductive}.  Let~\(\F\) be a non-Archimedean local field with a discrete valuation~\(v\).  We normalise~\(v\) by \(v(\F^\times) = \Z\).  Let \(\integ \subset \F\) be the ring of integers and \(\maxid \subset \integ\) its maximal ideal.  The cardinality~\(q\) of the residue field \(\integ / \maxid\) is a power of a prime number~\(p\).  We briefly call~\(\F\) a \(p\)\nb-adic field.

Bruhat and Tits~\cites{Bruhat-Tits:Reductifs_I, Bruhat-Tits:Reductifs_II, Tits:Reductive} constructed an affine building for any reductive \(p\)\nb-adic group \(\G = \Galg(\F)\).  More precisely, they constructed two buildings, one corresponding to~\(\G\) and one corresponding to the maximal semisimple quotient of~\(\G\).  We call the latter the Bruhat--Tits building of~\(\G\) and denote it by~\(\bt{\F}\).  Relying on \cite{Schneider-Stuhler:Rep_sheaves}*{\S\,1.1} and \cite{Vigneras:Cohomology}*{Section~1}, we now recall its construction.  The main ingredients are certain subgroups~\(\Un_{\alpha,r}\) and~\(H_r\) of~\(\G\).

\subsection{The prolonged valuated root datum}
\label{sec:prolonged_root}

Let \(\inp{\cdot}{\cdot}\colon X_*(\Salg) \times X^*(\Salg) \to \Z\) be the canonical pairing.  There is a unique group homomorphism
\[
\nu\colon \Ce(\ST) \to X_*(\Salg) \otimes_\Z \R
\]
such that \(\inp{\nu (z)}{\chi |_\ST} = - v (\chi (z))\) for all \(\chi \in X^*\bigl(\Ce(\ST)\bigr)\).  Let
\[
H \defeq \ker (\nu) = \bigl\{ z \in \Ce(\ST) :
\text{\(v(\chi (z)) = 0\)  for all \(\chi \in X^*(\Ce(\ST))\)} \bigr\}.
\]
be the maximal compact subgroup of~\(\Ce(\ST)\).

Bruhat and Tits~\cite{Bruhat-Tits:Reductifs_II} defined discrete decreasing filtrations of \(H\) and~\(\Un_\alpha\) by compact open subgroups \(H_r\) and~\(\Un_{\alpha,r}\), respectively.  These groups satisfy the properties of a ``prolonged valuated root datum'' \cite{Bruhat-Tits:Reductifs_I}*{\S\,6.2}.  We first describe these subgroups in the special case where~\(\Galg\) splits over~\(\F\).  Then each~\(\Un_\alpha\) is a one-dimensional vector space over~\(\F\), and a Chevalley basis of \(\Lie(\Galg)\) gives rise to an isomorphism \(\Un_\alpha \cong \F\).  Chevalley bases are known to exist but they are not unique.  We fix one, and we use suitable subsets as bases of \(\Lie(\Palg_D)\) and \(\Lie(\Lalg_D)\), for any standard parabolic subgroup~\(\Pa_D\) with Levi factor~\(\Le_D\).  Thus~\(\Un_\alpha\) is endowed with a discrete valuation~\(v_\alpha\) and one defines
\begin{equation}
  \label{eq:Usplit}
  \Un_{\alpha ,r} \defeq v_\alpha^{-1} ([r,\infty]) \qquad \text{for } r \in \R.
\end{equation}
By assumption, the maximal split torus is a maximal torus, that is, \(\Salg = \Calg(\Salg)\).  For \(r<0\) we may put \(H_r = H\), but~\(H_0\) is more difficult to define.  According to \cite{Bruhat-Tits:Reductifs_II}*{5.2.1} there is a canonical smooth affine \(\integ\)\nb-group scheme~\(\Zinteg\) such that \(\Zinteg(\F) = \Ce(\ST)\).  Let~\(\Zinteg_c\) be the neutral component of~\(\Zinteg\) and put \(H_0 \defeq \Zinteg_c (\integ)\).  The inclusions
\[
H_0 \subseteq \Zinteg(\integ) \subseteq H
\]
are all of finite index.  We define
\begin{equation}
  \label{eq:Hsplit}
  H_r \defeq \bigl\{ z \in H_0 :
  \text{\(v (\chi (z) - 1) \ge r\) for all \(\chi \in X^*(\Calg(\Salg))\)} \bigr\}
\end{equation}
for \(r>0\) as in \cite{Schneider-Stuhler:Rep_sheaves}*{Proposition I.2.6}.

Now we extend the above construction to a non-split group~\(\G\).  Proposition~\ref{prop:splitG} provides
a finite Galois extension~\(\tilde{\F}\) of~\(\F\) over which~\(\Galg\) splits.  The strategy of descent
is explained in \cite{Bruhat-Tits:Reductifs_I}*{Chapitre 9}; the basic idea is to construct the required
groups first in \(\Galg(\tilde{\F})\) and then to intersect them with \(\Galg(\F)\).  This does not work
as such because the root system of \(\Galg(\tilde{\F})\) is usually larger than that of \(\Galg(\F)\),
so that must be taken into account as well.  Bruhat and Tits descend in two steps: first from split to
quasi-split, then from there to the general case.  This is, in all probability, necessary for the proof,
but the conclusions can be written down in one step.  Of course it is by no means obvious that the
groups we will define below form a (prolonged) valuated root datum: proving this is precisely what
most of the work in~\cite{Bruhat-Tits:Reductifs_II} is dedicated to.

If~\(X\) is any object constructed over~\(\F\), then we will denote the corresponding object over~\(\tilde{\F}\) by~\(\tilde{X}\).  According to \cite{Serre:Local_fields}*{Proposition I.2.3} \(\tilde{\F}\) is also a local field, and there is a unique discrete valuation \(\tilde{v}\colon \tilde{\F} \to \Q \cup \{\infty\}\) that extends~\(v\).  By definition,
\[
\tilde{v} (\tilde{\F}^\times ) = \eF^{-1} \Z,
\]
where \(\eF \in \N\) is the ramification index of~\(\tilde{\F}\) over~\(\F\).  The constructions above still work for this non-normalised valuation~\(\tilde{v}\).

Let \(\tilde{\ST} \subseteq \Galg(\tilde{\F})\) be a maximal \(\tilde \F\)\nb-split torus that contains \(\Salg(\tilde{\F})\).  Since \(\tilde{\ST} \supseteq \ST\), restriction of characters defines a surjection
\begin{equation}
  \label{eq:rhoS}
  \rho_\ST\colon \tilde \Phi\cup \{0\} \to \Phi\cup \{ 0 \}.
\end{equation}
For \(\alpha \in \Phr\) and \(r \in \R\) the descent \cite{Bruhat-Tits:Reductifs_II}*{4.2.2 and 5.1.16} boils down to
\begin{equation}
  \label{eq:Ualphar}
  \begin{aligned}
    \Un_{\alpha ,r} &\defeq \Un_\alpha \cap \biggl( \prod_{\beta \in \rho_\ST^{-1}\{\alpha \}}
      \tilde{\Un}_{\beta ,r} \times \prod_{\beta \in \rho_\ST^{-1} \{2 \alpha \}}
      \tilde{\Un}_{\beta ,2r} \biggr) , \\
    \Un_{2 \alpha ,r} &\defeq \Un_{2 \alpha} \cap \Un_{\alpha ,r/2}.
  \end{aligned}
\end{equation}
These groups do not depend on the chosen ordering of the factors.  For a standard Levi subgroup \(\Le_D \subseteq \G\) and \(\alpha \in \Phi_D\), our consistent choice of Chevalley bases ensures that it does not matter whether we consider the groups~\(\Un_{\alpha,r}\) in \(\G\) or~\(\Le_D\).

We can use~\eqref{eq:Ualphar} to define a valuation on~\(\Un_\alpha\) by
\begin{equation}
  \label{eq:valpha}
  v_\alpha(u_\alpha) \defeq \sup {}\{ r \in \R : u_\alpha \in \Un_{\alpha,r}\}.
\end{equation}
Clearly this reproduces~\eqref{eq:Usplit} in the split case.
Let~\(\Gamma_\alpha\) be the set of \(r \in \R\) at which~\(\Un_{\alpha,r}\) jumps, or equivalently
the set of values of~\(v_\alpha\) (except \(v_\alpha (1) = \infty\)).  By construction,
\(\tilde{\Gamma}_\beta = \eF^{-1} \Z\) for all \(\beta \in \tilde \Phi\), which implies
\[
\Z \subseteq \Gamma_\alpha \subseteq \eF^{-1} \Z
\qquad \text{for all } \alpha \in \Phi.
\]
More precisely, \cite{Bruhat-Tits:Reductifs_I}*{6.2.23} and \cite{Schneider-Stuhler:Rep_sheaves}*{Lemma I.2.10} yield \(n_\alpha \in \N\) for \(\alpha \in \Phi\) with the following properties:
\begin{itemize}
\item \(\Gamma_\alpha = n_\alpha^{-1} \Z\);
\item \(n_{w \alpha} = n_\alpha\) for \(w \in W (\Phi)\);
\item \(n_{2 \alpha} = n_{\alpha}\) or \(n_{2 \alpha} = n_{\alpha} / 2\)
whenever \(\alpha, 2\alpha \in \Phi\).
\end{itemize}
Similar to \eqref{eq:Ualphar} one defines for \(r \in \R\)
(see \cite{Schneider-Stuhler:Rep_sheaves}*{I.2.6} and \cite{Vigneras:Cohomology}*{Section 1}) :
\begin{equation}
  \label{eq:Hr}
  H_r \defeq \Ce(\ST) \cap
  \biggl( \tilde{H}_r \times \prod_{\beta \in \rho_\ST^{-1}\{ 0 \}}
  \tilde{\Un}_{\beta ,r} \biggr).
\end{equation}
A particularly useful property of the above groups, which holds more or less by the definition of a 
prolonged valuated root datum \cite{Bruhat-Tits:Reductifs_I}*{Proposition 6.4.41}, is as follows.
Let \(\alpha,\beta \in \Phi \cup \{ 0 \}\) and let \(r,s \in \R\), with \(r \geq 0\) if \(\alpha = 0\)
and \(s \geq 0\) if \(\beta = 0\).  Then
\begin{equation}
  \label{eq:CommutatorGroups}
  [\Un_{\alpha,r},\Un_{\beta,s}] \subseteq \text{subgroup generated by }
  \bigcup_{n,m \in \Z_{>0}} \Un_{n \alpha + m \beta, nr+ms},
\end{equation}
where \(\Un_{0,t} = H_t\) and \(\Un_{\delta,t} = \{1\}\) if \(\delta \notin \Phi \cup \{0\}\).
We will need an iterated version of this, which must have been known already to Bruhat and Tits, but
for which the authors did not find a reference.

\begin{lemma} \label{lem:multipleCommutators}
Let $\alpha_i \in \Phi^+ \cup \{0\}, r_i \in \R$ and $u_i \in \Un_{\alpha_i ,r_i}$ for $i=1,2,\cdots,n$. 
Assume that $r_i \geq 0$ whenever $\alpha_i = 0$. Then
\[
[u_1,[u_2,[\cdots [u_{n-1},u_n] \cdots ]]]
\] 
lies in the group generated by the $\Un_{\sum_{i=1}^n k_i \alpha_i,
\sum_{i=1}^n k_i r_i}$, where the $k_i$ run over $\Z_{>0}$.
\end{lemma}
\begin{proof}
Let us call the group in question $K$. Suppose that $y_j \in \Un_{\sum_{i=2}^n k_i \alpha_i,
\sum_{i=2}^n k_i r_i}$ for some $k_i \in \Z_{>0}$ (depending on $j$). Notice that
$\sum_{i=2}^n k_i \alpha_i$ cannot be a negative root, and that $\sum_{i=2}^n k_i r_i \geq 0$
if $\sum_{i=2}^n k_i \alpha_i = 0$.
We will show by induction on $l \in \N$ that 
\[
[u_1, y_1 \cdots y_l] \text{ is an element of } K.
\]
For $l = 1$ this is \eqref{eq:CommutatorGroups}. For $l \geq 2$ we can rewrite it as
\[
[u_1, y_1 \cdots y_l] = u_1 y_1 u_1^{-1} [u_1, y_2 \cdots y_l ] y_1^{-1} = [u_1,y_1] y_1 [u_1,y_2 \cdots y_l] y_1^{-1} .
\]
By the induction hypothesis all terms on the right are in $K$.

For the actual lemma we use another induction, with respect to $n$. The case $n=1$ is trivial. For $n>1$, 
the induction hypothesis provides $y_j$ as above, such that
\[
[u_1,[u_2,[\cdots [u_{n-1},u_n] \cdots ]]] = [u_1, y_1 \cdots y_l] ,
\]
which by the above lies in $K$.
\end{proof}

\subsection{The affine Bruhat--Tits building}
\label{sec:def_building}

The image of any cocharacter \(\F^\times \to \Zc\) lies in \(\ST_\Delta \subseteq \ST\), the maximal \(\F\)\nb-split torus in~\(\Zc\).  Hence \(X_*(\Zc) = X_*(\ST_\Delta)\).  The standard \emph{apartment} is
\[
\Apa \defeq \bigl( X_*(\ST) / X_*(\Zc) \bigr) \otimes_\Z \R =
\bigl( X_*(\Salg) / X_*(\Salg_\Delta) \bigr) \otimes_\Z \R.
\]
The affine Bruhat--Tits building \(\bt{\F}\) will be defined as \(\G \times \Apa \mathbin/ {\sim}\) for a suitable equivalence relation~\(\sim\).

Let \(\inp{\cdot}{\cdot}_{\Apa}\) be a \(W(\Phi)\)-invariant inner product on~\(\Apa\).  Then the different irreducible components~\(\Phi_i^\vee\) of~\(\Phi^\vee\) are orthogonal and on~\(\R \Phi_i^\vee\) the inner product is unique up to scaling.  Thus we may assume that \(\inp{\alpha^\vee}{\alpha^\vee}_{\Apa} = 1\) for all short coroots \(\alpha^\vee \in \Phi^\vee\).

The centraliser~\(\Ce(\ST)\) acts on~\(\Apa\) by
\[
g \cdot x = x + \nu (g).
\]
This extends to an action of~\(\No(\ST)\) on~\(\Apa\) by affine automorphisms, such that the linear part of \(x \mapsto g \cdot x\) is given by the image of \(g \in \No(\ST)\) in \(W(\Phi)\).  In particular, the action of~\(g\) on~\(\Apa\) is a translation if and only if \(g \in \Ce(\ST)\).  The affine hyperplanes
\begin{equation}
  \Apa[\ST,\alpha ,k] \defeq \{ x \in \Apa : \inp{x}{\alpha} = k \} \qquad
  \text{for \(\alpha \in \Phi\) and \(k \in \Gamma_\alpha\)}
\end{equation}
turn~\(\Apa\) into a polysimplicial complex.  The open polysimplices are called facets, that is, a \emph{facet} in~\(\Apa\) is a non-empty subset \(F \subseteq \Apa\) such that
\begin{itemize}
\item \(F \subseteq \Apa[\ST,\alpha ,k] \) or~\(F\) lies entirely on one side of~\(\Apa[\ST,\alpha ,k] \) for all \(\alpha \in \Phi\) and \(k \in \Gamma_\alpha\);
\item \(F\) cannot be extended to a larger set with the first property.
\end{itemize}
Thus the closure of a facet is a polysimplex, and a facet is closed if and only if it is a single point.  Moreover, a facet is open in~\(\Apa\) if and only if it is of maximal dimension, in which case we call it a \emph{chamber}.

The affine action of~\(\No(\ST)\) on~\(\Apa\) respects the polysimplicial structure.  In fact, \(\No(\ST)\) is generated by the translations coming from~\(\Ce(\ST)\) and the reflections in the hyperplanes~\(\Apa[\ST,\alpha,k]\):
\[
x \mapsto x + (k - \inp{x}{\alpha}) \alpha^\vee \qquad
\alpha \in \Phi, k \in \Gamma_\alpha,
\]
where \(\alpha^\vee \in \Phi^\vee\) is the coroot corresponding to~\(\alpha\).

For a non-empty subset \(\Omega \subseteq \Apa\) we define
\begin{equation}
  \label{eq:fOmega}
  f_\Omega\colon \Phi\to \R \cup \{\infty\},\qquad
  f_\Omega (\alpha ) \defeq - \inf_{x \in \Omega} {}\inp{x}{\alpha}
  = \sup_{x \in \Omega} {}\inp{x}{-\alpha}.
\end{equation}
This gives rise to the following subgroups of~\(\G\):
\begin{equation}
  \label{eq:POmega}
  \begin{aligned}
    \UC &\defeq \text{subgroup generated by } \bigcup\nolimits_{\alpha \in \Phr}
    \Un_{\alpha, f_\Omega (\alpha)} ,\\
    \NC &\defeq \{ n \in \No(\ST) : \text{\(n \cdot x = x\) for all \(x \in \Omega\)} \} ,\\
    \PC &\defeq \NC \UC = \UC \NC.
  \end{aligned}
\end{equation}
The latter is a group because \(n \UC n^{-1} = \UC[n \Omega]\) for all \(n \in \No(\ST)\).  For \(\Omega = \{x\}\) we abbreviate \(\UC = \UC[x]\), which should not be confused with the root subgroups~\(\Un_\alpha\).

Given a partition \(\Phi= \Phi^+ \cup \Phi^-\) of \(\Phi(\Galg,\Salg)\) in positive and negative roots, we let~\(\Un^\pm\) be the subgroup of~\(\G\) generated by \(\bigcup_{\alpha \in \Phi^\pm} \Un_\alpha\).  We write
\[
\UC^+ \defeq \UC \cap \Un^+ \quad \text{and} \quad
\UC^- \defeq \UC \cap \Un^-.
\]

\begin{proposition}[\cite{Bruhat-Tits:Reductifs_I}*{6.4.9}]
  \label{prop:decompose}
  These subgroups have the following properties:
  \begin{enumerate}[label=\textup{(\alph{*})}]
  \item\label{prop:decomposea} \(\UC \cap \Un_\alpha = \Un_{\alpha, f_\Omega (\alpha)}\) for all \(\alpha \in \Phi\).
  \item\label{prop:decomposeb} The product map
    \[
    \prod_{\alpha \in \Phr \cap \Phi^\pm} \Un_{\alpha, f_\Omega (\alpha)} \to \UC^\pm
    \]
    is an isomorphism of algebraic varieties, for any ordering of the factors.
  \item\label{prop:decomposec} \(\UC = \UC^+ \UC^- \bigl( \UC \cap \No(\ST) \bigr)\).
  \end{enumerate}
\end{proposition}

We define an equivalence relation~\(\sim\) on \(\G \times \Apa\) by
\[
(g,x) \sim (h,y)  \iff
\text{there is \(n \in \No(\ST)\) with \(n x = y\)  and \(g^{-1} h n \in \UC[x]\).}
\]
As announced, the Bruhat--Tits building of~\(\G\) is
\[
\bt{\F} = \G \times \Apa \mathbin/ {\sim}.
\]
The group~\(\G\) acts naturally on~\(\bt{\F}\) from the left, and the map
\[
\Apa \to \bt{\F},\qquad x \mapsto (1,x) \mathbin/ {\sim}
\]
is an \(\No(\ST)\)\nb-equivariant embedding.  An apartment of~\(\bt{\F}\) is a subset of the form \(g \cdot \Apa\) with \(g \in \G\), and \(g\cdot \Apa=\Apa\) if and only if \(g\in \No(\ST)\).  Since all maximal split tori of~\(\G\) are conjugate by \cite{Borel-Tits:Groupes_reductifs}*{Th\'eor\`eme 4.21}, there is a bijection between apartments in \(\bt{\F}\) and maximal split tori in~\(\G\).

A facet of~\(\bt{\F}\) is a subset of the form~\(g \cdot F\), where \(g \in \G\) and~\(F\) is a facet of~\(\Apa\).  For a polysimplicial complex~\(\Sigma\), we denote the set of vertices by~\(\Sigma^\circ\) and the set of \(n\)\nb-dimensional polysimplices in~\(\Sigma\) by~\(\Sigma^n\) for \(n\in\N\).

For any subset \(\Omega \subseteq \bt{\F}\), we denote the pointwise stabiliser of~\(\Omega\) by~\(\PC\).  This is consistent with~\eqref{eq:POmega} when \(\Omega \subseteq \Apa\).

\section{Fixed points in the building}
\label{sec:fixedpoints}

An element~\(g\) of~\(\G\) is called \emph{compact} if its image in \(\G / \Ze\) belongs to a compact subgroup of \(\G / \Ze\).  According to the Bruhat--Tits Fixed Point Theorem (see~\cite{Bruhat-Tits:Reductifs_I}*{\S\,3.2}), the compact elements of \(\G\) are precisely those that fix a point in the building \(\bt{\F}\).  In this section, we study how the fixed point subset \(\bt{\F}^\gamma\) depends on~\(\gamma\).

Let~\(H\) be a group of polysimplicial automorphisms of~\(\bt{\F}\).  If \(x,y \in \bt{\F}^H\), then~\(H\) fixes the geodesic segment~\([x,y]\) pointwise by \cite{Bruhat-Tits:Reductifs_I}*{2.5.4}.  Consequently, \(\bt{\F}^H\) is a convex subset of \(\bt{\F}\).  Recall that a \emph{chamber complex} is a polysimplicial complex~\(\Sigma\) such that:
\begin{itemize}
\item all maximal polysimplices of~\(\Sigma\) (the chambers) have the same dimension;
\item given any two chambers \(C_1\) and~\(C_2\) of~\(\Sigma\), there exists a gallery of chambers connecting \(C_1\) and~\(C_2\).
\end{itemize}
If \(g\in\G\) is compact and belongs to a maximal split torus~\(\ST\) of~\(\G\), then there is a chamber in the corresponding apartment~\(\Apa\) that is fixed pointwise by~\(g\).  There exist, however, regular semisimple elements \(\gamma\in\G\) that fix no chamber in the building pointwise.  For such elements the fixed point subcomplex is not necessarily a chamber complex.  But once~\(g\) fixes a chamber, say, because it belongs to a maximal split torus, the fixed point subset is automatically a chamber complex:

\begin{lemma}
  \label{lem:chamberComplex}
  Suppose that~\(H\) fixes a chamber \(C \subseteq \bt{\F}\) pointwise.  Then~\(\bt{\F}^H\) is a chamber complex.
\end{lemma}

\begin{proof}
  This is well-known, but we include a proof anyway.  Let \(x \in \bt{\F}^H\) and let~\(\Apa[x]\) be an apartment that contains \(C\) and~\(x\).  Since \(\dim C = \dim \Apa[x]\) and \(\bt{\F}^H\) is convex, it contains an open subset of some chamber \(C_x \subseteq \Apa[x]\) with \(x\in\cl{C_x}\).  Thus~\(H\) fixes~\(C_x\) pointwise and \(\bt{\F}^H\) is the union of all its closed chambers.

  Suppose that~\(\mathcal C\) is any collection of chambers of an apartment~\(\Apa\) of \(\bt{\F}\).  Then \(\bigcup_{C\in\mathcal C} \cl{C}\) is convex if and only if all minimal galleries between elements of~\(\mathcal C\) are contained in~\(\mathcal C\).  Hence \(\bt{\F}^H \cap \Apa\) contains all minimal galleries between its chambers.
\end{proof}

\subsection{The split case}
\label{sec:fixed_split}

Let \(\ST \subset \G\) be a split maximal torus and let \(\gamma \in \ST\) be a compact element.  Then \(v \bigl(\chi(\gamma)\bigr) = 0\) for all \(\chi \in X^* (\ST)\), so that~\(\gamma\) fixes the apartment~\(\Apa\) pointwise.  The subcomplex \(\bt{\F}^\gamma \subseteq \bt{\F}\) is convex and \(\ST\)\nb-invariant.  Its core is formed by the apartment~\(\Apa\) and from there ``hairs'' extend in all directions.  This terminology applies quite well to one-dimensional buildings, but in general such a hair is a (not necessarily bounded) chamber complex.  Since~\(\ST\) acts by translations on~\(\Apa\), it shifts all these hairs.  If \(\gamma \in \ST\) is regular, then \(\bt{\F}^\gamma / \ST \) is compact by \cite{Korman:Character}*{Section 9.1}: the length of the hairs is finite.

Now we study when an arbitrary point \(x \in \bt{\F}\) is fixed by \(\gamma \in \ST\).  Choose a chamber \(C_0 \subseteq \Apa\) and let~\(\rho\) be the retraction of \(\bt{\F}\) to~\(\Apa\) centred at~\(C_0\).  Let~\(\Phi^+\) be a system of positive roots in~\(\Phi\) such that \(f_{\rho(x)}(\alpha) \geq f_{C_0}(\alpha)\) for all \(\alpha \in \Phi^+\); equivalently, \(\Phi^+\) contains all roots with \(f_{\rho(x)}(\alpha) > f_{C_0}(\alpha)\).  Let~\(\Delta\) be the basis of~\(\Phi\) corresponding to~\(\Phi^+\).

Then \(\UC[C_0] \cap \Un_\alpha \subseteq \UC[\rho(x)] \cap \Un_\alpha\) for all \(\alpha \in \Phi_-\), so \(\UC[C_0]^- \subseteq \UC[\rho (x)]^-\).  Furthermore, \(\NC[C_0] = \NC[\rho (x)]\), which together with Proposition~\ref{prop:decompose}.\ref{prop:decomposec} shows that \(\PC[C_0] \subseteq \UC[C_0]^+ \PC[\rho (x)]\).  Since~\(\PC[C_0]\) acts transitively on the set of apartments containing~\(C_0\) by \cite{Bruhat-Tits:Reductifs_I}*{7.4.9}, there is \(u \in \UC[C_0]^+\) with \(x = u \rho (x)\).  Thus we want to know which part of the apartment~\(u \Apa\) is fixed by~\(\gamma\).

By definition, \(u \in \UC[C_0]^+\) fixes all \(y \in \Apa\) satisfying \(-\alpha (y) \leq f_{C_0}(\alpha)\) for all \(\alpha \in \Phi^+\).  These points constitute a cone in \(\Apa \cap u \Apa\), which is fixed by~\(\gamma\).  We are interested in the larger subset \((u \Apa)^\gamma\), which is a convex subcomplex of \(\bt{\F}^\gamma\).  Hence the complex \(Y \defeq u^{-1} (u \Apa)^\gamma\) is convex as well.  Concretely, this means that \(Y \subseteq \Apa\) is determined by a system of equations \(-\alpha (y) \leq r_\alpha\) for certain \(r_\alpha \in \R\), \(\alpha \in \Phi^+\).  We need some notation to make this more explicit.  The \emph{singular depth} of~\(\gamma\) in the direction \(\alpha \in \Phi\) is
\[
\sd_\alpha (\gamma) \defeq v (\alpha (\gamma) - 1).
\]
We also let \(\sd (\gamma) \defeq \max\limits_{\alpha \in
  \Phi^+} \sd_\alpha (\gamma)\).

Recall that the height of a positive root is defined as follows:
\begin{itemize}
\item \(\hgt(\alpha) = 1\) if \(\alpha \in \Phi^+\) is simple;
\item \(\hgt(\alpha + \beta) = \hgt(\alpha) + \hgt(\beta)\) if \(\alpha, \beta, \alpha+\beta \in \Phi^+\).
\end{itemize}
Since~$\hgt$ extends to a group homomorphism $X^* (\G / \Zc)
\to \R$, we may regard it as a point in the apartment~$\Apa$.
Since~\(y\) is contained in the same apartment, this gives
meaning to the linear combination \(y+\sd(\gamma)\hgt\) for
\(y\in Y\) appearing in Proposition~\ref{prop:fixpoints}.(c)
below.

By Proposition~\ref{prop:decompose}.\ref{prop:decomposeb} we can write
\begin{equation}
  \label{eq:prodUalpha}
  u = \prod_{\alpha \in \Phi^+} u_\alpha \qquad\text{with}\quad u_\alpha \in \Un_{\alpha, f_{C_0}(\alpha)}.
\end{equation}

\begin{proposition}
  \label{prop:fixpoints}
  Let \(y \in u^{-1} (u \Apa)^\gamma\).
  \begin{enumerate}[label=\textup{(\alph{*})}]
  \item The compact element \(\gamma \in \ST\) fixes \(x = u \rho (x)\) if and only if \([\gamma,u^{-1}] \in \UC[\rho (x)]^+\).
  \item \(u_\alpha \in \Un_{\alpha, -\alpha (y) - \sd_\alpha (\gamma)} \) for all simple roots \(\alpha \in \Delta\).
  \item \(u \in \Un^+_{y + \sd (\gamma) \hgt }\), where $\sd
    (\gamma) \hgt \in \Apa$.
  \end{enumerate}
\end{proposition}
\begin{proof}
  (a) Since \(\gamma \in \ST\) fixes \(\rho (x) \in \Apa\),
  \[
  \gamma (x) = \gamma u \rho (x) = \gamma u \gamma^{-1} \rho (x).
  \]
  This point equals \(x = u \rho(x)\) if and only if \(\gamma u^{-1} \gamma^{-1} u \rho (x) = \rho (x)\), which is equivalent to \([\gamma ,u^{-1}] \in \PC[\rho (x)]\).  As \(u \in \Un^+\) and~\(\gamma\) normalises~\(\Un^+\), this is equivalent to \([\gamma,u^{-1}] \in \PC[\rho (x)] \cap \Un^+ = \UC[\rho (x)]^+\).

  (b) The decomposition~\eqref{eq:prodUalpha} is unique once we fix an ordering on~\(\Phi^+\), but the terms~\(u_\alpha\) may depend on this ordering.  Let \(\Phi^* \defeq \Phi^+ \setminus \Delta\) be the set of non-simple positive roots.  Then \(\bigcup_{\alpha \in \Phi^*} (\Un_{\alpha} \cap \UC[C_0])\) generates a normal subgroup~\(\UC[C_0]^*\) of~\(\UC[C_0]^+\).  The quotient \(\UC[C_0]^+ / \UC[C_0]^*\) is abelian and can be identified with a lattice in the \(\F\)\nb-vector space \(\prod_{\alpha \in \Delta} \Un_{\alpha} \).  The image of~\(u\) in \(\UC[C_0]^+ / \UC[C_0]^*\) is \(\prod_{\alpha \in \Delta} u_\alpha\), which shows that the ingredients~\(u_\alpha\) of~\eqref{eq:prodUalpha} for \(\alpha\in\Delta\) are independent of the ordering of~\(\Phi^+\).

  Suppose now that~\(\gamma\) fixes \(u y \in u \Apa\).  By part~(a), we have \([\gamma, u^{-1}] \in \UC[y]^+\).  Since~\(\gamma\) normalises the groups \(\Un_{\alpha,r}\) for \(\alpha \in \Phi^+\), \(r \in \R\), this implies
  \begin{equation}
    \label{eq:prodCommutators}
    [\gamma,u^{-1}] \UC[y]^* = \prod_{\alpha \in \Delta} [\gamma,u_\alpha^{-1}] \UC[y]^* \in \UC[y]^+ / \UC[y]^*.
  \end{equation}
  But on the vector space \(\Un_{\alpha} \) the map \(a \mapsto [\gamma ,a]\) can be identified with multiplication by \(\alpha (\gamma) - 1\).  Hence~\eqref{eq:prodCommutators} is equivalent to
  \begin{equation}
    \label{eq:valualpha}
    u_\alpha \in (\alpha (\gamma) - 1)^{-1} \Un_{\alpha,-\alpha (y)}
  \end{equation}
  for all \(\alpha \in \Delta\). Together with \eqref{eq:Usplit} implies the statement~(b).

  (c) We fix an ordering \(\Phi^+ = \{\alpha_1 ,\alpha_2, \dotsc, \alpha_k \}\) with 
$\hgt (\alpha_{i}) \leq \hgt (\alpha_{i+1})$ for all $i$, and we get a unique decomposition 
\(u = \prod_{i=1}^k u_{\alpha_i}\) in~\(\UC[C_0]^+\).  Similarly,
Proposition~\ref{prop:decompose}.\ref{prop:decomposeb} yields a unique decomposition
  \begin{equation}
    \label{eq:factorCommutator}
    \prod_{i=1}^k [\gamma ,u^{-1}]_{\alpha_i} = [\gamma ,u^{-1}] =
    \gamma u_{\alpha_k}^{-1} u_{\alpha_{k-1}}^{-1} \cdots u_{\alpha_1}^{-1} \gamma^{-1}
    u_{\alpha_1} u_{\alpha_2} \cdots u_{\alpha_k} .
  \end{equation}
  By construction \([\gamma ,u^{-1}]_{\alpha} \in \Un_{\alpha, f_{C_0} (\alpha)}\), and~\(\gamma\) fixes~\(uy\) if and only if, even more,
  \begin{alignat}{2}
    \label{eq:condAlpha}
    [\gamma ,u^{-1}]_{\alpha} &\in \Un_{\alpha,- \alpha (y)} &\qquad&\text{for all \(\alpha \in \Phi^+\).}\\
\intertext{Assuming~\eqref{eq:condAlpha}, we will show by induction on \(\hgt(\alpha)\) that}
    \label{eq:valuationCommutator}
    [\gamma, u_\alpha^{-1}] &\in \Un_{\alpha, - \alpha (y) + (1- \hgt(\alpha)) \sd (\gamma)} 
    &\qquad&\text{for all \(\alpha \in \Phi^+\).}
  \end{alignat}
Like in \eqref{eq:valualpha}, this statement is equivalent to 
$u_\alpha \in \Un_{\alpha, - \alpha (y) - \sd_\alpha (\gamma) + (1- \hgt(\alpha)) \sd (\gamma)}$, which
for roots~\(\alpha\) of height~\(1\) is part (b).

  Let us assume~\eqref{eq:valuationCommutator} for roots of height less than~\(k\).  Let~\(N_{>k}\) be the product of the groups~\(\Un_\alpha\) for roots~\(\alpha\) of height greater than~\(k\).  This is a normal subgroup of the Borel group \(\ST \Un^+\), and the subgroups \(\Un_\alpha\subseteq\Un^+\) for a root~\(\alpha\) of height~\(k\) become central in the quotient \(\Un^+/N_{>k}\).  We may determine the \(\alpha\)\nb-component for a root~\(\alpha\) of height~\(k\) by computations in~\(\Un^+/N_{>k}\) because of the uniqueness of the decomposition~\eqref{eq:prodUalpha}.

   Now we split~\(u\) up as \(u_{<k} u_k u_{>k}\), where the factors \(u_{<k}\), \(u_k\) and~\(u_{>k}\) contain the contributions~\(u_\alpha\) of positive roots~\(\alpha\) with height less than~\(k\), equal to~\(k\), and greater than~\(k\), respectively.  In the quotient~\(\Un^+/N_{>k}\), we may drop~\(u_{>k}\), and~\(u_k\) becomes central.  Hence
\begin{multline} \label{eq:gammau-1}
   [\gamma,u^{-1}]
   = \gamma u_{>k}^{-1} u_k^{-1} u_{<k}^{-1} \gamma^{-1} u_{<k} u_k u_{>k}
   \equiv \gamma u_k^{-1} \gamma^{-1} \gamma u_{<k}^{-1} \gamma^{-1} u_{<k} u_k \\
   \equiv \gamma u_k^{-1} \gamma^{-1} u_k \gamma u_{<k}^{-1} \gamma^{-1} u_{<k}
   = [\gamma,u_k^{-1}] [\gamma,u_{<k}^{-1}] 
   \equiv \Big( \prod_{\hgt (\alpha ) = k} [\gamma, u_\alpha^{-1}] \Big) [\gamma,u_{<k}^{-1}] ,
\end{multline}
where we compute in the quotient~\(\Un^+/N_{>k}\). We will use the induction hypothesis and 
the estimate on \([\gamma,u]_\alpha\) to estimate \([\gamma,u_\alpha^{-1}]\) when $\hgt (\alpha) = k$.

We first rewrite a commutator \([\gamma,z_1\dotsm z_j]\) as a product of iterated commutators
\begin{equation}
    \label{eq:termsBetaC}
    C(z_{i_1},\dotsc,z_{i_k}) \defeq [z_{i_1}, [z_{i_2}, \dotsc, [z_{i_{k-1}}, [\gamma,z_{i_k}]]\dotsc]].
\end{equation}
  We claim that \([\gamma,z_1\dotsm z_j]\) is a product of the factors \(C(z_{i_1},\dotsc,z_{i_k})\) with \(1\le i_1< i_2< \dotsb < i_k\le j\), each factor appearing exactly once.  The proof is by induction on~\(j\), the case \(j=1\) being clear.  For the induction step, we use
  \begin{align*}
    [\gamma,z_1\dotsm z_j] &= \gamma z_1 \gamma^{-1} [\gamma,z_2\dotsm z_j] z_1^{-1},\\
    \gamma z_1 \gamma^{-1} x_1\dotsm x_k z_1^{-1}
    &= [\gamma,z_1] \cdot [z_1,x_1]x_1 \cdot [z_1,x_2]x_2 \dotsm [z_1,x_k]x_k.  
  \end{align*}
  By the induction hypothesis, \([\gamma,z_2\dotsm z_j]\) is the product in some order of the factors \(C(z_{i_1},\dotsc,z_{i_k})\) for all \(2\le i_1<\dotsb<i_k\le j\).  Plugging this into the second equation above shows that \([\gamma,z_1\dotsm z_j]\) is the product in some order of the factors \(C(z_{i_1},\dotsc,z_{i_k})\) for all \(1\le i_1<\dotsb<i_k\le j\).  By the way, a more careful induction argument also yields the order of the factors: it is the reverse lexicographic order for the words \((j-i_k, i_{k-1}, i_{k-2},\dotsc,i_1)\).

Now we apply this to $u_{<k}^{-1} = u_{\alpha_l}^{-1} \cdots u_{\alpha_1}^{-1} =  z_1 \cdots z_l$.
By the induction hypothesis and by Lemma \ref{lem:multipleCommutators}, all the occurring 
$C (u^{-1}_{\alpha_{i_1}}, \cdots u^{-1}_{\alpha_{i_k}} )$ lie in the group generated by the $\Un_{\alpha,r}$ 
with $\alpha = \sum_{j=1}^k k_j \alpha_{i_j}$ and $r = \sd_{\alpha_{i_k}} (\gamma) + \sum_{j=1}^k k_j r_{i_j}$, 
where $k_j \in \Z_{>0}$ and 
\[
r_{i_j} = -\alpha_{i_j}(y) - \sd_{\alpha_{i_j}}(\gamma) + (1 - \hgt (\alpha_{i_j})) \sd (\gamma) .
\]
For such $\alpha \in \Phi^+$ and $r \in \R$ we have
\begin{multline} \label{eq:valTerms}
r = \sd_{\alpha_{i_k}} (\gamma) + \sum_{j=1}^k k_j \bigl( - \alpha_{i_j} (y) - \sd_{\alpha_{i_j}} (\gamma) +
    (1 - \hgt (\alpha_{i_j}))) \sd (\gamma) \bigr) \geq
    \\ - \alpha (y) + (1 - \hgt(\alpha)) \sd (\gamma) + \biggl( -1 + \sum_{j=1}^k k_j \biggr) \bigl( \sd (\gamma) - 
  \max_j \sd_{\alpha_{i_j}} (\gamma) \bigr)  \geq \\ 
- \alpha (y) + (1 - \hgt(\alpha)) \sd (\gamma).
\end{multline}
For a root $\alpha$ of height $k$, \eqref{eq:condAlpha} and \eqref{eq:gammau-1} show that
$[\gamma,u_\alpha^{-1}]$ must lie in the largest of the groups $\Un_{\alpha,-\alpha (y)}$ and 
$\Un_{\alpha,r}$. Now we see from \eqref{eq:valTerms} that in any case $[\gamma,u_\alpha^{-1}] \in
\Un_{\alpha,- \alpha (y) + (1 - \hgt(\alpha)) \sd (\gamma)}$, so
\[
u_\alpha^{-1},u_\alpha \in \Un_{\alpha, - \alpha (y) - \hgt(\alpha) \sd (\gamma)} =
\Un_\alpha \cap U^+_{y + \sd (\gamma) \hgt}. \qedhere
\]
\end{proof}

Given an arbitrary point \(y \in \Apa\), the condition in Proposition~\ref{prop:fixpoints}.(c) does not imply that~\(\gamma\) fixes~\(u y\).  Counterexamples exist whenever~\(\Phi\) contains an irreducible root system of rank greater than one.

Proposition~\ref{prop:fixpoints} only applies to fixed points of semisimple elements that lie in a split maximal torus.  (We will not consider the fixed points of non-semisimple elements of~\(\G\) in this article.)  For elements of non-split maximal tori we need yet another aspect of Bruhat--Tits theory.

\subsection{The non-split case}

The construction of the Bruhat--Tits building over \(p\)\nb-adic fields is functorial with respect to finite field extensions by \cite{Bruhat-Tits:Reductifs_I}*{9.1.17}.  For any such extension \(\tilde{\F} / \F\), the group
\[
\Gamma \defeq \{ \sigma \in \Aut(\tilde{\F}) : \sigma |_\F = \id_\F\}
\]
acts naturally on \(\bt{\tilde{\F}}\), and \(\bt{\F}\) is contained in \(\bt{\tilde{\F}}^\Gamma\).  In particular, for every \(g \in \Galg (\F)\) we have an inclusion
\begin{equation}
  \label{eq:Galoisinv}
  \bt{\F}^g = \bt{\tilde{\F}}^g \cap \bt{\F}
  \subseteq \bt{\tilde{\F}}^{\Gamma \times \gen{g}},
\end{equation}
where \(\gen{g} \subseteq \Galg (\tilde{\F})\) denotes the subgroup generated by~\(g\).

In general, \(\bt{\F}\) is strictly smaller than \(\bt{\tilde{\F}}^\Gamma\), even if \(\tilde{\F} / \F\) is a Galois extension (in which case~\(\Gamma\) is its Galois group).  Rousseau~\cite{Rousseau:Thesis} proved that \(\bt{\F} = \bt{\tilde{\F}}^\Gamma\) if \(\tilde{\F} / \F\) is a tamely ramified Galois extension, see also~\cite{Prasad:Galois-fixed}.  Consequently, \eqref{eq:Galoisinv} is an equality for such extensions.

Let \(\T = \Talg (\F)\) be a maximal torus and \(\tilde{\F} / \F\) a finite Galois extension over which~\(\Talg\) splits, as in Proposition~\ref{prop:splitG}.  Since~\(\Talg\) is defined over~\(\F\), it is \(\Gamma\)\nb-stable, and hence the corresponding apartment~\(\tApa{\Talg(\tilde{\F})}\) of \(\bt{\tilde{\F}}\) is \(\Gamma\)\nb-stable.  The action of \(\Gamma\) on \(\tApa{\Talg(\tilde{\F})}\) is linear, so that the origin of~\(\tApa{\Talg(\tilde{\F})}\) is fixed.  Thus Rousseau's above result implies that
\begin{equation}
  \label{eq:Rousseau}
  \bt{\F} \cap \tApa{\Talg (\tilde{\F} )} \neq \emptyset
  \quad\text{if \(\tilde{\F} / \F\) is tamely ramified.}
\end{equation}

Any \(g \in \G\) acts on \(\Lie(\Galg) \bigm/ \Lie\bigl(\Calg(g)\bigr)\) by the adjoint representation.  The collection \(E(g)\) of eigenvalues (in some algebraic closure of~\(\F\)) is finite and does not contain~\(1\).  Assume that~\(\G\) is not a torus and that~\(g\) is regular, that is, \(\Calg(g)\) has the smallest possible dimension.  The number
\[
\sd(g) \defeq \max_{\lambda \in E(g)} v(\lambda - 1)
\]
is well-defined because every eigenvalue lies in a finite field extension of~\(\F\).  For irregular \(g \in \G\) we put \(\sd(g) = \infty\), because in that case the multiplicity of the eigenvalue~\(1\) of \(\Ad(g) \in \Endo_\F \bigl(\Lie(\Galg)\bigr)\) is too high.  Finally, if~\(\G\) is a torus, then we define \(\sd(g) = 0\) for all \(g \in \G\).  This definition stems from \cite{Adler-Korman:Local_character}*{Section 4}, where \(\sd(g)\) is called the \emph{singular depth} of~\(\gamma\).  We note that
\begin{equation} \label{eq:sfunction}
\sd (g z) = \sd(g) = \sd (h g h^{-1}) \qquad
\text{for \(z \in \Ze\) and \(h \in \G\).}
\end{equation}
Let \(\T\) and~\(\tilde{\F}\) be as above and let \(\tilde \Phi = \Phi\bigl(\Galg(\tilde{\F}), \Talg(\tilde{\F})\bigr)\) be the corresponding root system.  Let~\(\tilde{v}\) be the discrete valuation that extends~\(v\) and suppose \(\gamma \in \T\).  Then
\[
\sd(\gamma) = \max_{\alpha \in \tilde \Phi} \sd_\alpha (\gamma),
\]
which agrees with the notation from Proposition~\ref{prop:fixpoints}.(c).
Notice that \(\sd(\gamma) \ge 0\), for if \(\sd_\alpha (\gamma) < 0\) then \(\tilde{v} (\alpha (\gamma)) < 0\) , so \(\tilde{v} (\alpha (\gamma)^{-1}) > 0\) and \(\sd_{-\alpha}(\gamma) = 0\). 

Now we specialise to a compact regular semisimple element \(\gamma \in \T\).  Then \(\bt{\F}^\gamma\) is non-empty by the Bruhat--Tits Fixed Point Theorem.  If \(\T / \Zc\) is anisotropic, then \(\bt{\F}^\gamma\) is a finite polysimplicial complex (see \cite{Schneider-Stuhler:Rep_sheaves}*{p.~53}) and there is an open neighbourhood~\(\Un\) of~\(\gamma\) in~\(\G\) such that \(\bt{\F}^\Un = \bt{\F}^\gamma\).

If \(\T / \Zc\) is not anisotropic, we have a weaker substitute.  Since \(\bt{\F}^\gamma / \T\) is compact, there exists an open neighborhood~\(V\) of~\(\gamma\) in~\(\T\) such that \(\bt{\F}^g = \bt{\F}^\gamma\) for all \(g \in V\).  Let~\(\tilde{H}_r\) be as in~\eqref{eq:Hr}, but with respect to \(\bigl( \Galg(\tilde{\F}) ,\Talg(\tilde{\F}) \bigr)\).  First the authors believed that one could take \(V = \gamma \tilde H_r \cap \T\) for any \(r > \sd (\gamma)\), but this turns out to be incorrect in general.  We thank the referee for pointing out the weakness in our former argument.

\begin{lemma}
  \label{lem:constantfix}
  Write \(\hgt (\Phi) \defeq \max_{\alpha \in \Phi^+} \hgt (\alpha)\) and let \(r > \hgt (\Phi) \sd(\gamma)\).  Then \(\bt{\F}^{\gamma h} = \bt{\F}^\gamma\) for all \(h \in \tilde H_r \cap \T\).
\end{lemma}

\begin{proof}
  In view of \eqref{eq:Galoisinv} it suffices to prove the corresponding statement for fixed points in the building \(\bt{\tilde \F}\).  We use the notation from the proof of Proposition \ref{prop:fixpoints}, but with some additional tildes.  We want to know when~\(\gamma\) fixes~\(u y\), for some point \(y \in\tilde{\Apa}\).  According to~\eqref{eq:condAlpha}, this is equivalent to
  \begin{equation}
    \label{eq:condTildeAlpha}
    [\gamma ,u^{-1}]_\alpha \in \tilde \Un_{\alpha, -\alpha (y)} \qquad \text{for all }
    \alpha \in \tilde \Phi.
  \end{equation}
From \eqref{eq:gammau-1} we know that apart from $[\gamma, u_\alpha^{-1}]$, all the contributions to 
\([\gamma ,u^{-1}]_\alpha\) come from commutators of elements $u_\beta^{-1}$ with $\hgt (\beta) < \hgt (\alpha)$.
Supposing that~\(u_\beta\) has already been fixed for all roots~\(\beta\) of smaller height than $\alpha$, 
\eqref{eq:condTildeAlpha} determines which \(u_\alpha \in \tilde \Un_\alpha\) can give rise to fixed 
points~\(u y\).

  Recall from Section~\ref{sec:prolonged_root} that we have a Chevalley basis of \(\Lie[\tilde \F](\Galg)\) and corresponding isomorphisms of algebraic groups \(\tilde \Un_\alpha \cong \tilde \F\).  
These restrict to
\[
\tilde \Un_{\alpha,r} \cong \{ \lambda \in \tilde \F : \tilde v (\lambda) \geq r \}
\qquad \text{for all } r \in \R ,
\]
and if $u_\alpha$ corresponds to  $\lambda_\alpha \in \tilde{\F}$, then \([\gamma ,u_\alpha^{-1}] \) becomes \((1 - \alpha (\gamma)) \lambda_\alpha\).
Because we are interested in~\(u y\), the component~\(u_\alpha\) is determined only modulo \(\tilde \Un_{\alpha,-\alpha (y)}\), that is, \(\lambda_\alpha\) modulo \(\{ \lambda \in \tilde \F : \tilde v (\lambda) \geq -\alpha (y) \}\) is all that matters.

Now we compare~\(\gamma\) with~\(\gamma h\). We note that for all $\beta \in \Phi$
\begin{multline} \label{eq:valGammah}
  \tilde v \bigl( (1 - \beta (\gamma)) - (1 - \beta (\gamma h)) \bigr) = \tilde v \bigl(
  \beta (\gamma) (\beta (h) - 1) \bigr) \\=
  \tilde v (\beta (h) - 1) = \sd_\beta (h) \geq r > \hgt (\Phi) \sd (\gamma) .
\end{multline}
By~\eqref{eq:valTerms} the valuation of a contribution from $C (u^{-1}_{\alpha_{i_1}}, \cdots u^{-1}_{\alpha_{i_k}} )$ to $[\gamma, u^{-1}]_\alpha$ is at least
\begin{equation} \label{eq:valTermsL}
-\alpha (y) + (1 - \hgt(\alpha)) \sd (\gamma) .
\end{equation}
Recall that $C (u^{-1}_{\alpha_{i_1}}, \cdots u^{-1}_{\alpha_{i_k}} )$ also involves 
$[\gamma,u^{-1}_{i_k}]$. If we use $\gamma h$ instead of $\gamma$, then by \eqref{eq:valGammah} 
and \eqref{eq:valTermsL} we get a new element whose $v_\alpha$-value  differs only in the fractional 
ideal of~\(\tilde \F\) where the valuation is at least
\[
-\alpha (y) + (1 - \hgt(\alpha)) \sd (\gamma) + \hgt(\Phi) \sd (\gamma)
\geq -\alpha (y) + \sd (\gamma) .
\]
So, if the \(u_\beta\) with \(\hgt(\beta) < \hgt(\alpha)\) have already been fixed, then the condition 
\eqref{eq:condTildeAlpha} for both $\gamma$ and $\gamma h$ leads to two sets of solutions for 
\(\lambda_\alpha\), and these sets differ only in the parts of valuation at least
\[
-\alpha (y) + \sd (\gamma) - \sd_\alpha (\gamma) \geq -\alpha (y).
\]
But these parts do not influence the point~\(u y\).  Hence~\(\gamma h\) fixes such a point~\(u y\) if and only if~\(\gamma\) does.  Since this holds for all \(y \in \tilde \Apa\) we conclude that
\[
\bt{\tilde \F}^{\gamma h} = \bt{\tilde \F}^\gamma .  \qedhere
\]
\end{proof}

\section{The groups \texorpdfstring{$\UC^{(e)}$}{UOmega(e)}}
\label{sec:Usigmae}

Schneider and Stuhler introduced an important system of compact subgroups of~\(\G\), which they used to derive several interesting results on complex smooth \(\G\)\nb-representations in~\cite{Schneider-Stuhler:Rep_sheaves}.  These subgroups were also studied by Moy and Prasad in \cites{Moy-Prasad:Unrefined, Moy-Prasad:Jacquet} for their theory of unrefined minimal types, and by Vign\'eras in~\cite{Vigneras:Cohomology} in the context of \(\G\)\nb-representations on vector spaces over general fields.

Let~\(\tilde{\R}\) be the set \(\R \cup \{ r{+} : r \in \R\} \cup \{\infty\}\) endowed with the ordering
\[
r < r{+} < s < s{+} < \infty \qquad \text{if \(r<s\).}
\]
We define addition and multiplication with positive numbers on~\(\tilde{\R}\) in the obvious way, so that they respect the ordering.  For example
\[
r + (s{+}) = (r+s){+} \quad\text{and}\quad 2 \cdot r{+} = (2r){+}.
\]
Starting with the filtrations \eqref{eq:Ualphar} and~\eqref{eq:Hr} we define for \(\alpha \in \Phi\) and \(r \in \R\):
\begin{equation}\label{eq:defUa+}
  \begin{alignedat}{2}
    \Un_{\alpha, r+} & \defeq \bigcup_{s > r} \Un_{\alpha,s}, &\qquad
    \Un_{\alpha,\infty} & \defeq\{1\} ,\\
    H_{r+} & \defeq \bigcup_{s > r} H_s, &\qquad
    H_\infty & \defeq \{1\}.
  \end{alignedat}
\end{equation}
Since the filtrations are discrete, we have \(\Un_{\alpha,r+} = \Un_{\alpha,r + \epsilon}\) for sufficiently small \(\epsilon > 0\), and similarly for~\(H_{r+}\).

For a function \(f\colon \Phi\cup \{0\} \to \tilde{\R}\), let~\(\UC[f]\) be the subgroup of~\(\G\) generated by \(\bigcup_{\alpha\in\Phi} \Un_{\alpha,f(\alpha)} \cup H_{f(0)}\).  For non-empty \(\Omega \subseteq \Apa\) we vary on~\eqref{eq:fOmega} by
\begin{equation}
  \label{eq:f*Omega}
  f^*_\Omega\colon \Phi\cup \{0\} \to \tilde{\R},\qquad
  \alpha \mapsto
  \begin{cases}
    \inp{\Omega}{-\alpha}{}+ &\text{if~\(\alpha\) is constant on~\(\Omega\),}\\
    \sup_{x\in\Omega} {}\inp{x}{-\alpha} &\text{otherwise.}
  \end{cases}
\end{equation}
For \(e \in \R_{\ge 0}\), we define
\[
\UC^{(e)} \defeq \Un_{f_\Omega^*+e}.
\]
Notice that the closure~\(\cl{\Omega}\) of~\(\Omega\) yields \(f^*_{\cl{\Omega}} = f^*_\Omega\) and hence \(\UC[\cl{\Omega}]^{(e)} = \UC^{(e)}\).

\begin{example}
  \label{exa:U_for_GL}
  Let \(\G = \GL_n (\F)\).  We identify the standard apartment~\(\Apa\) of \(\bt[\GL_n]{\F}\) with \(\R^n \mathbin/ \R (1,1,\dotsc, 1)\), such that the set of vertices is the image of~\(\Z^n\).  Denote the smallest integer larger than \(r{+} \in \tilde{\R}\) by~\(\ceil{r}\).  Recall the fractional ideals~\(\maxid^m\) in~\(\F\) for \(m \in \Z\).  For a point \(x = (x_1,\dotsc,x_n) \in \Apa\) and \(e\in\R_{\ge0}\) we have
  \[
  \UC[x]^{(e)} =
  \left(\begin{gathered}\xymatrix@-2.6em{
        1 + \maxid^{\ceil{e}} \ar@{.}[ddrr]&
        \maxid^{\ceil{x_2 - x_1 + e}} \ar@{.}[rrr] \ar@{.}[dddrrr]&&&
        \maxid^{\ceil{x_n - x_1 + e}} \ar@{.}[ddd]\\
        \maxid^{\ceil{x_1 - x_2 + e}}\ar@{.}[dddrrr]\ar@{.}[ddd]\\
        &&1 + \maxid^{\ceil{e}}\ar@{.}[ddrr]\\
        &&&&\maxid^{\ceil{x_n - x_{n-1} + e}}\\
        \maxid^{\ceil{x_1 - x_n + e}}\ar@{.}[rrr]&&&
        \maxid^{\ceil{x_{n-1} - x_n + e}} & 1 + \maxid^{\ceil{e}}
    }\end{gathered}\right)
  \]
  If \(e\in\Z_{\ge0}\) and \(\Omega \subset \Apa\) is the standard chamber, defined by \(x_1>x_2>\dotsb>x_n>x_1-1\), then
  \[
  \UC^{(e)} =
  \left(\begin{gathered}\xymatrix@-2em{
        1 + \maxid^{e+1} \ar@{.}[ddrr]&
        \maxid^e \ar@{.}[rrr] \ar@{.}[dddrrr]&&&
        \maxid^e \ar@{.}[ddd]\\
        \maxid^{e+1}\ar@{.}[dddrrr]\ar@{.}[ddd]\\
        &&1 + \maxid^{e+1}\ar@{.}[ddrr]\\
        &&&&\maxid^e\\
        \maxid^{e+1}\ar@{.}[rrr]&&&
        \maxid^{e+1} & 1 + \maxid^{e+1}
    }\end{gathered}\right).
  \]
  Notice that \(\UC^{(0)}\) is contained in the standard Iwahori subgroup of \(\GL_n (\F)\), and that they are not equal because the diagonal entries differ.
\end{example}

The groups \(\UC^{(e)}\) satisfy the following \emph{unique decomposition property}.

\begin{proposition}[\cite{Bruhat-Tits:Reductifs_I}*{6.4.48}]
  \label{prop:udp}
  For any ordering of~\(\Phr\) the product map
  \[
  H_{e{+}} \times \prod_{\alpha \in \Phr} (\UC^{(e)} \cap \Un_\alpha)
  \to \UC^{(e)}
  \]
  is a diffeomorphism.  Moreover \(\UC^{(e)} \cap \No(\ST) = H_{e{+}}\) and for \(\alpha \in \Phr\)
  \[
  \UC^{(e)} \cap \Un_\alpha =
  \begin{cases}
    \Un_{\alpha ,f_\Omega^* (\alpha) + e} &\text{if \(2\alpha \notin \Phi\),}\\
    \Un_{\alpha ,f_\Omega^* (\alpha) + e} \cdot
    \Un_{2\alpha ,f_\Omega^* (2\alpha) + e} &\text{if \(2\alpha \in \Phi\).}
  \end{cases}
  \]
\end{proposition}
By a diffeomorphism between \(p\)\nb-adic algebraic varieties we mean a homeomorphism~\(f\), such that \(f\) and~\(f^{-1}\) are given locally by convergent power series.  The above product map is obviously algebraic, but its inverse need not be.

There is a version of the unique decomposition property with \(\Phr \cup \{0\}\) instead of~\(\Phr\).  It follows easily from Proposition~\ref{prop:udp}, since~\(H_{e{+}}\) normalises~\(\Un_{\alpha,r}\).

The above decomposition implies that the subgroups~\(\UC^{(e)}\) behave well with respect to field extensions and Levi subgroups.

\begin{lemma}
  \label{lem:Uetildea}
  Let \(\tilde{\F} / \F\) be a finite field extension and let \(\tUC^{(e)}\subseteq\Galg(\tilde{\F})\) be defined like \(\UC^{(e)}\subseteq\Galg(\F)\).  Then \(\UC^{(e)} = \tUC^{(e)} \cap \Galg (\F)\).
\end{lemma}

\begin{proof}
  Let~\(\tilde{\ST}\) and~\(\rho_\ST\) be as on page~\pageref{eq:rhoS} and let \(\tApa{\tilde{\ST}} \supseteq \Apa\) be the corresponding apartment of \(\bt{\tilde{\F}}\).  Then \(\tilde{f}^*_\Omega (\alpha) = f^*_\Omega (\rho_\ST (\alpha))\) for all \(\alpha \in \tilde{\Phi}\).  Now apply Proposition~\ref{prop:udp} and Equations \eqref{eq:Ualphar} and~\eqref{eq:Hr}.
\end{proof}

Let \(\Le_D = \Lalg_D (\F)\) be a standard Levi subgroup of~\(\G\).  Then a maximal split torus~\(\ST\) of~\(\G\) is a maximal split torus of~\(\Le_D\) as well, and the standard apartment of~\(\bt[\Lalg_D]{\F}\) is
\[
\Apa[D] \defeq \bigl( X_*(\ST) \bigm/ X_*(\Zc[\Le_D]) \bigr) \otimes_\Z \R =
\bigl( X_*(\Salg) \bigm/ X_*(\Salg_D) \bigr) \otimes_\Z \R.
\]
Since \(\ST_\Delta \subseteq \ST_D\), there is a quotient map between the apartments
\begin{equation}
  \label{eq:AxD}
  \Apa \to \Apa[\ST_D],\qquad x \mapsto x_D,
\end{equation}
in the buildings for \(\G\) and~\(\Le_D\).

\begin{lemma}
  \label{lem:Uetildeb}
  Let~\(\Omega_D\) be the image of~\(\Omega\) in the standard apartment~\(\Apa[D]\) of the building for~\(\Le_D\).  Then \(\UC[\Omega_D]^{(e)} = \UC^{(e)} \cap \Le_D\) and
  \[
  \UC^{(e)} = \bigl( \UC^{(e)} \cap \Ru[\Pa_D] \bigr)
  \bigl( \UC^{(e)} \cap \Le_D \bigr)
  \bigl( \UC^{(e)} \cap \Ru[\bar{\Pa}_D] \bigr).
  \]
\end{lemma}

\begin{proof}
  For \(\Omega \subseteq \Apa\) and \(\alpha\in\Phi_D\) we clearly have \(f^*_{\Omega_D}(\alpha) = f^*_\Omega (\alpha)\).  As the groups \(\Un_{\alpha,r}\) and~\(H_r\) are the same in~\(\Le_D\) and in~\(\G\), the statement follows from Proposition~\ref{prop:udp}.
\end{proof}

We are mainly interested in the cases where~\(\Omega\) is a point, a facet or a polysimplex.

\begin{theorem}
  \label{thm:UOmegae}
  For a point~\(x\), a polysimplex~\(\sigma\), and a general subset~\(\Omega\) of an apartment~\(\Apa\), the following hold:
  \begin{enumerate}[label=\textup{(\alph{*})}]
  \item\label{UOmegae_0} \(\UC^{(e)}\) is open if~\(\Omega\) is bounded.
  \item\label{UOmegae_1} \(\UC^{(e)}\) is compact.
  \item\label{UOmegae_2} \(\UC^{(e)}\) is normal in~\(\PC\).
  \item\label{UOmegae_3} \(\UC[x]^{(e)}\) fixes the star of~\(x\) pointwise.
  \item\label{UOmegae_4} \(\UC[\sigma]^{(e)} = \prod_{\text{\(x\) vertex of \(\sigma\)}} \UC[x]^{(e)}\) if \(e \in \Z_{\ge 0}\).
  \item\label{UOmegae_5} If \(x\) is an interior point of~\(\sigma\) and \(e\in\Z_{\ge 0}\), then \(\UC[x]^{(e)} = \UC[\sigma]^{(e)}\).
  \item\label{UOmegae_6} \(\UC^{(e)} \supseteq \UC^{(e')}\) whenever \(e \le e'\).
  \item\label{UOmegae_7} The groups \(\UC[\sigma]^{(e)}\) for \(e \in \N\) form a neighbourhood basis of~\(1\) in~\(\G\).
  \item\label{UOmegae_8} The group generated by the commutators \(\bigl[\UC^{(e)},\UC^{(e')}\bigr]\) is contained in~\(\UC^{(e+e')}\).
  \end{enumerate}
\end{theorem}

Since \(\Un_{\alpha,r} = \{1\}\) if and only if \(r = \infty\), \ref{UOmegae_0} follows from Proposition \ref{prop:udp}.  Statements \ref{UOmegae_2} and~\ref{UOmegae_3} show that the order of the product in~\ref{UOmegae_4} does not matter.  The proofs of \ref{UOmegae_1}--\ref{UOmegae_4} and \ref{UOmegae_6}--\ref{UOmegae_7} may be found in \cite{Schneider-Stuhler:Rep_sheaves}*{Section I.2}.  Property~\ref{UOmegae_5} is \cite{Vigneras:Cohomology}*{Proposition 1.1}, whereas~\ref{UOmegae_8} follows from \cite{Bruhat-Tits:Reductifs_I}*{6.4.41}.  Notice that so far these properties hold only for subsets of the standard apartment~\(\Apa\).  However, \ref{UOmegae_2}~allows us to define
\begin{equation}
  \label{eq:UOmegae}
  \UC^{(e)} \defeq g \, \UC[g^{-1} \Omega]^{(e)} \, g^{-1}
\end{equation}
for any non-empty subset~\(\Omega\) of an apartment~\(g \Apa\).  Now Theorem~\ref{thm:UOmegae} holds in the entire building \(\bt{\F}\).

We need one more important property.  We define the hull~\(\Hull (\sigma ,\tau)\) of two polysimplices \(\sigma\) and~\(\tau\) as the intersection of all apartments containing \(\sigma\cup\tau\).  This finite polysimplicial complex is a combinatorial approximation to the closed convex hull of \(\sigma\cup\tau\).  Similarly, we can define the hull \(\Hull(x,z)\) of two arbitrary points \(x,z \in \bt{\F}\).  The proof of \cite{Vigneras:Cohomology}*{Lemma 1.28} yields
\begin{enumerate}[resume,label=\textup{(\alph{*})}]
\item\label{UOmegae_9} If \(x,z \in \bt{\F}\) and \(y\in\Hull(x,z)\), then \(\UC[y]^{(e)} \subseteq \UC[x]^{(e)} \UC[z]^{(e)}\).
\end{enumerate}

The fixed points of the groups~\(\Un_{\alpha,k}\) in the standard apartment are described by \cite{Bruhat-Tits:Reductifs_I}*{7.44}:
\begin{equation}
  \label{eq:AU}
  \begin{aligned}
    \Apa^{\Un_{\alpha ,k}} &=
    \{ x \in \Apa : \inp{x}{\alpha} \ge -k \} , \\
    \Apa^{\Un_{\alpha ,k{+}}} &=
    \{ x \in \Apa : \inp{x}{\alpha} \ge -k - n_\alpha^{-1} \}.
\end{aligned}
\end{equation}
for all \(\alpha \in \Phi\) and \(k \in \Gamma_\alpha\).  Let \(\lfloor r \rfloor_{\Gamma_\alpha}\) for \(r \in \R\) denote the largest element of~\(\Gamma_\alpha\) that is \emph{strictly} smaller than~\(r\).  For \(x \in \Apa\), \eqref{eq:AU}, Proposition~\ref{prop:udp} and Theorem~\ref{thm:UOmegae}.\ref{UOmegae_2} yield
\begin{equation}
  \begin{aligned}
    \Apa^{\UC[x]^{(e)}} &= \{ y \in \Apa :
    \text{\(\inp{y}{\alpha} \ge \lfloor \alpha (x) - e \rfloor_{\Gamma_\alpha}\) for all \(\alpha\in\Phi\)}\},\\
    \bt{\F}^{\UC[x]^{(e)}} &= \PC[x] \cdot \Apa^{\UC[x]^{(e)}}.
  \end{aligned}
\end{equation}

\subsection{The level of representations}
\label{sec:representation}

The system of subgroups \(( \UC[x]^{(e)})_{x \in \bt{\F}^\circ}\) for fixed \(e\in\Z_{\ge0}\) is a ``consistent equivariant system of subgroups'' in the terminology of \cite{Meyer-Solleveld:Resolutions}*{\S\,2.2} because of properties \ref{UOmegae_1}, \ref{UOmegae_4}, and~\ref{UOmegae_9} in Theorem~\ref{thm:UOmegae} and~\eqref{eq:UOmegae}.  The main result of \cite{Meyer-Solleveld:Resolutions}, which was inspired by \cite{Korman:Character}*{Section 7.1}, uses these subgroups to construct resolutions of \(\G\)\nb-representations and suitable subsets thereof.  We now describe this in greater detail.

Let~\(\repr\) be a representation of~\(\G\) on a \(\Z[1/p]\)\nb-module~\(V\), where~\(p\) is the characteristic of the residue field of~\(\F\).  For any polysimplicial subcomplex \(\Sigma \subseteq \bt{\F}\) we define
\[
\Cell[n] (\Sigma; V) \defeq \bigoplus_{\sigma\in\Sigma^n}
V^{\UC[\sigma]^{(e)}} \otimes_\Z \Z\{\sigma\}.
\]
If~\(\tau\) is a face of~\(\sigma\), then \(\UC[\tau]^{(e)} \subseteq \UC[\sigma]^{(e)}\) by Theorem~\ref{thm:UOmegae}.\ref{UOmegae_4} above, so that \(V^{\UC[\tau]^{(e)}} \supseteq V^{\UC[\sigma]^{(e)}}\).  Fix any orientation of \(\bt{\F}\) and declare~\(\sigma\) endowed with the opposite orientation to be equal to \(-\sigma \in \Z\{\sigma\}\).  We define a boundary map
\begin{equation}
  \partial_n\colon \Cell[n](\Sigma; V) \to \Cell[n-1](\Sigma; V),
  \qquad v \otimes \sigma\mapsto v \otimes \partial(\sigma).
\end{equation}
Here \(\partial(\sigma)\) is the usual boundary of~\(\sigma\), a weighted sum of codimension-one faces of~\(\sigma\).  This yields a chain complex \(\bigl(\Cell(\Sigma;V),\partial_*\bigr)\), that is, \(\partial^2= 0\).  We augment it by
\begin{equation}
  \partial_0\colon \Cell[0](\Sigma; V) \to V,\qquad
  v \otimes x \mapsto v.
\end{equation}
If \(g \in \G\) and \(g\cdot\Sigma \subseteq \Sigma\), then~\(g\) acts on \(\Cell(\Sigma ;V)\) by
\[
g \cdot (v\otimes\sigma) = \repr (g) v \otimes g\cdot\sigma ,
\]
where~\(g\cdot\sigma\) is endowed with the orientation coming from~\(\sigma\).

\begin{theorem}[\cite{Meyer-Solleveld:Resolutions}*{Theorem 2.4}]
  \label{thm:resolution}
  Let~\(\Sigma\) be a convex subcomplex of~\(\bt{\F}\), let \(e\in\Z_{\ge 0}\), and let \(\repr\colon \G\to\Aut(V)\) be a representation as above.  Then \(\bigl(\Cell(\Sigma; V), \partial_*\bigr)\) is exact in all positive degrees, and the augmentation map~\(\partial_0\) induces a bijection
  \[
  \Ho_0(\Sigma;V) \cong \sum_{x\in\Sigma^\circ} V^{\UC[x]^{(e)}}.
  \]
\end{theorem}

\begin{definition}
  \label{def:level}
  A (smooth) \(\G\)\nb-representation~\(V\) has \emph{level} \(e\in\Z_{\ge 0}\) if
  \[
  V = \sum_{x\in\bt{\F}^\circ} V^{\UC[x]^{(e)}}.
  \]
\end{definition}

This level is similar to the depth of a representation defined by Vign\'eras in \cite{Vigneras:l-modulaires}*{II.5.7}, generalising~\cite{Moy-Prasad:Unrefined}.  More precisely, if~\(V\) is irreducible and \(e\) is the smallest integer such that~\(V\) has level~\(e\), then the depth of~\(V\) lies in \((e-1,e]\).  The category of \(\G\)\nb-representations of level~\(e\) is studied in \cite{Meyer-Solleveld:Resolutions}*{Section~3}.  If~\(V\) is a complex \(\G\)\nb-representation of level~\(e\) and \(\Sigma = \bt{\F}\), then Theorem~\ref{thm:resolution} recovers a result of Schneider and Stuhler \cite{Schneider-Stuhler:Rep_sheaves}*{II.3.1}.  As we will see later, Theorem~\ref{thm:resolution} for finite subcomplexes has independent significance.

Let~\(\Pa\) be a parabolic subgroup of~\(\G\) with unipotent radical~\(\Ru\).  We let
\[
V(\Ru) \defeq \Span \{\repr(g) v - v : g \in \Ru\},\qquad
V_{\Ru} \defeq V \bigm/ V(\Ru).
\]
The representation \((\repr_{\Ru} ,V_{\Ru})\) of \(\Pa\) or~\(\Pa/\Ru\) is called the (unnormalised) \emph{parabolic restriction} of~\(V\).

Let \((\rho ,W)\) be a smooth representation of \(\Pa / \Ru\).  Inflate it to a representation of~\(\Pa\) and construct the smoothly induced \(\G\)\nb-representation \(\Ind_\Pa^\G(W)\).  This is known as the (unnormalised) \emph{parabolic induction} of~\(W\).

\begin{proposition}
  \label{prop:levele}
  Let \(\Pa \subseteq \G\) be a parabolic subgroup.
  \begin{enumerate}[label=\textup{(\alph{*})}]
  \item If~\(V\) is a \(\G\)\nb-representation of level~\(e\), then~\(V_{\Ru}\) is a representation of \(\Pa / \Ru\) of level~\(e\).
  \item If~\(W\) is a representation of \(\Pa / \Ru\) of level~\(e\), then \(\Ind_\Pa^\G(W)\) has level~\(e\).
  \end{enumerate}
\end{proposition}

\begin{proof}
  We first establish~(a).  We may assume that \(\Pa = \Pa_D\) is a standard parabolic subgroup.  Then \(\Un^+ \subseteq \Pa_D\).  \cite{Bruhat-Tits:Reductifs_I}*{Proposition 7.3.1} yields \(\G = \Pa_D \No(\ST) \UC[C]\) for any chamber \(C \subseteq \Apa\).  Since~\(\cl C\) is a fundamental domain for the action of~\(\G\) on~\(\bt{\F}\),
  \[
  \bt{\F}^\circ
  = \G \cdot \cl{C}^\circ
  = \Pa_D \No(\ST) \UC[C] \cl{C}^\circ
  = \Pa_D \No(\ST) \cl{C}^\circ
  = \Pa_D \Apa^\circ.
  \]
  The definition of the level and Lemma~\ref{lem:Uetildeb} yield
  \begin{multline*}
    V = \sum_{x \in \bt{\F}^\circ} V^{\UC[x]^{(e)}}
    = \sum_{p \in \Pa_D} \sum_{x \in \Apa^\circ} p \cdot V^{\UC[x]^{(e)}}
    \\\subseteq \sum_{p \in \Pa_D} \sum_{x \in \Apa^\circ}
    p \cdot V^{\UC[x]^{(e)} \cap \Le_D}
    = \sum_{p \in \Pa_D} \sum_{x_D \in \Apa[D]^\circ} p \cdot V^{\UC[x_D]^{(e)}}.
  \end{multline*}
  This implies that~\(V_{\Ru[\Pa_D]}\) has level~\(e\) as well:
  \[
  V_{\Ru[\Pa_D]} = \sum_{p \in \Pa_D} \sum_{x_D \in \Apa[D]^\circ}
  p \cdot V_{\Ru[\Pa_D]}^{\UC[x_D]^{(e)}}
  = \sum_{x_D \in \bt[\Lalg_D]{\F}^\circ}
  V_{\Ru[\Pa_D]}^{\UC[x_D]^{(e)}}.
  \]

  Now we establish~(b).  For notational convenience, we assume that \(\Pa = \Pa_D\) is standard parabolic, so that we may identify \(\Pa / \Ru\) with \(\Le_D = \Lalg_D (\F)\).  A representation of~\(\Le_D\) has level~\(e\) if and only if it is a quotient of a direct sum of copies of the regular representation on \(\Ccinf(\Le_D/\UC[x_D]^{(e)})\) for points~\(x_D\) in the building of~\(\Le_D\); here~\(\Ccinf\) denotes the space of locally constant functions with compact support.  Since Jacquet induction preserves direct sums and quotients, it suffices to prove that the Jacquet induction of \(\Ccinf(\Le_D/\UC[x_D]^{(e)})\) has level~\(e\).  Inspection shows that this Jacquet induction is isomorphic to the regular representation on \(\Ccinf(\G/\Ru[\Pa_D]\UC[x_D]^{(e)})\).

  The subgroup~\(\Ru[\Pa_D]\UC[x_D]^{(e)}\) of~\(\G\) is an inductive limit of compact subgroups because~\(\UC[x_D]^{(e)}\) is compact and~\(\Ru[\Pa_D]\) is unipotent.  It is useful to choose a special sequence of compact subgroups exhausting~\(\Ru[\Pa_D]\), namely,
  \[
  \CG_n \defeq \gamma^n (\UC[x_D]^{(e)}\cap \Ru[\Pa_D])\gamma^{-n},
  \]
  where~\(\gamma\) is a central element of~\(\Le_D\) that is strictly positive, that is, \(\bigcup \CG_n = \Ru[\Pa_D]\).  We also consider the subgroups \(\oppa{\CG}_n \defeq \gamma^n (\UC[x_D]^{(e)}\cap \Ru[\oppa{\Pa}_D])\gamma^{-n}\) in the opposite unipotent group; then \(\bigcap \oppa{\CG}_n = \{1\}\).

  The space \(\Ccinf(\G/\Ru[\Pa_D]\UC[x_D]^{(e)})\) is the coinvariant space for the right action of \(\Ru[\Pa_D]\UC[x_D]^{(e)}\) on \(\Ccinf(\G)\).  This coinvariant space for an increasing union of compact subgroups is the inductive limit
  \[
  \Ccinf(\G/\Ru[\Pa_D]\UC[x_D]^{(e)})
  \cong \varinjlim \Ccinf(\G/\CG_n\UC[x_D]^{(e)})
  \cong \varinjlim \Ccinf\bigl(\G \bigm/ \gamma^n (\UC[x]^{(e)} \cap \Pa_D)\gamma^{-n}\bigr).
  \]
  Here~\(x\) is a pre-image of~\(x_D\) in the building for~\(\G\) for the map in~\eqref{eq:AxD}.  Thus \(\UC[x]^{(e)}\cap \Le_D = \UC[x_D]^{(e)}\) and
  \[
  \UC[x]^{(e)} =
  (\UC[x]^{(e)}\cap \Ru[\Pa_D]) \cdot 
  (\UC[x]^{(e)}\cap \Le_D) \cdot
  (\UC[x]^{(e)}\cap \Ru[\oppa{\Pa}_D]).
  \]
  Any smooth compactly supported function on \(\G/\gamma^n (\UC[x]^{(e)} \cap \Pa_D)\gamma^{-n}\) is invariant under right translation by~\(\oppa{\CG}_m\) for sufficiently large~\(m\) because \(\bigcap \oppa{\CG}_m=1\).  Hence we may rewrite
  \begin{multline*}
    \Ccinf(\G/\Ru[\Pa_D]\UC[x_D]^{(e)})
    \cong \varinjlim_{n,m} \Ccinf\bigl(\G \bigm/ \oppa{\CG}_m\gamma^n (\UC[x]^{(e)} \cap \Pa_D)\gamma^{-n}\bigr)
    \\\cong \varinjlim_{n} \Ccinf\bigl(\G \bigm/ \oppa{\CG}_n\gamma^n (\UC[x]^{(e)} \cap \Pa_D)\gamma^{-n}\bigr)
    \cong \varinjlim_{n} \Ccinf\bigl(\G \bigm/ \gamma^n \UC[x]^{(e)}\gamma^{-n}\bigr).
  \end{multline*}
  Since the regular representations on \(\Ccinf(\G / \gamma^n \UC[x]^{(e)}\gamma^{-n}) \cong \Ccinf(\G / \UC[x]^{(e)})\) have level~\(e\), so has their inductive limit.  Hence \(\Ccinf(\G/\Ru[\Pa_D]\UC[x_D]^{(e)})\) has level~\(e\) as asserted.
\end{proof}

\section{Characters of admissible representations}
\label{sec:admissible}

We define the character of an admissible representation first as a distribution and then describe how to interpret it as a locally constant function on suitable open subsets.  Our discussion is purely algebraic and also works for representations over arbitrary fields whose characteristic is different from the characteristic~\(p\) of the residue field of~\(\F\).

There is a Haar measure~\(\mu\) on~\(\G\) such that \(\mu (\CG) \in \Z [1/p]\) for all compact open subgroups \(\CG \subseteq \G\) by \cite{Meyer-Solleveld:Resolutions}*{Lemma 1.1}.  Let \(\Hec\) be the \(\Z [1/p]\)-module of locally constant functions \(\G \to \Z [1/p]\) with compact support.  Define the convolution product of \(f_1 ,f_2 \in \Hec\) by
\[
(f_1 * f_2) (h) = \int_\G f_1 (g) f_2 (g^{-1} h) \,\diff\mu(g).
\]
We call \(\Hec\) endowed with this multiplication the \emph{Hecke algebra}.  It is an associative idempotented, non-unital \(\Z [1/p]\)-algebra.  Every element of~\(\G\) naturally defines a multiplier of \(\Hec\), but is not contained in \(\Hec\).  Given a pro-\(p\) compact open subgroup \(\CG \subseteq \G\), we let
\[
\ide{\CG} = \mu(\CG)^{-1} 1_\CG \in \Hec
\]
be the corresponding idempotent.

A smooth representation~\(\repr\) of~\(\G\) on a \(\Z[1/p]\)-module~\(V\) becomes a \(\Hec\)-module in a natural way, and we have \(\ide{\CG} V = V^\CG\), the module of \(\CG\)\nb-invariant vectors in~\(V\).  We call an \(\Hec\)-module~\(W\) \emph{smooth} if \(W = \varinjlim {}\ide{\CG} W\), where the limit runs over all pro-\(p\) compact open subgroups~\(\CG\) of~\(\G\).  There is a natural equivalence between the following categories:
\begin{itemize}
\item smooth representations of~\(\G\) on \(\Z [1/p]\)-modules,
\item smooth \(\Hec\)-modules
\end{itemize}
(see \cite{Meyer-Solleveld:Resolutions}*{Proposition 1.3}).  We say that a representation~\(\G\) on a \(\K\)\nb-vector space~\(V\) has \emph{good characteristic} if the characteristic of the field~\(\K\) does not equal~\(p\).

In good characteristic, we may define the algebra \(\Hecke(\G,\K)\), whose smooth modules are in bijection with smooth representations of~\(\G\) on \(\K\)\nb-vector spaces.  Such a representation \((\repr,V)\) is called \emph{admissible} if~\(V^\CG\) has finite dimension for all compact open subgroups \(\CG \subseteq \G\).  An admissible representation in good characteristic gives rise to a distribution
\[
\theta_\repr\colon \Hecke(\G,\K) \to \K,\qquad
f \mapsto\tr(\repr(f),V).
\]

If \(\K=\C\), then Harish-Chandra's Theorem~\ref{thm:HC} shows that this distribution is associated to a locally integrable function, that is, \(\theta_\repr(f) = \int f(g)\cdot \trpi(g)\,\diff\mu(g)\) for all \(f\in\Hecke(\G,\C)\) and a locally integrable function~\(\trpi\).  Furthermore, \(\trpi\) is locally constant on the subset of regular semisimple elements.  Since this subset has full measure, the distribution~\(\theta_\repr\) is determined by the values of~\(\trpi\) on regular semisimple elements.  If~\(V\) has infinite dimension, then~\(\trpi\) is not locally constant near a unipotent element~\(u\) because the closure of the conjugacy class of~\(u\) contains~\(1\) and \(\trpi(1) = \dim V = \infty\).

Since integration requires analysis, the notion of a locally integrable function is unclear for a general field~\(\K\).  The following definition of a character function makes sense for any field~\(\K\):

\begin{definition}
  \label{def:trpi}
  Let \((\repr ,V)\) be an admissible \(\K\)\nb-linear representation of~\(\G\) and let \(g\in\G\).  We write \(\trpi(g)=\tau\in\K\) if there is a compact open subgroup~\(\CG\) such that \(\tr(\repr(f),V) = \tau\cdot \int_\G f(g)\,\diff\mu(g)\) for all \(f\in\Hec\) that are supported in \(\CG g \CG\).
\end{definition}

By definition, the domain of definition \(\dom\trpi\) of~\(\trpi\) is open in~\(\G\), and~\(\trpi\) is locally constant on \(\dom\trpi\).  Moreover, the trace property of~\(\theta_\repr\) forces the function~\(\trpi\) to be a class function, that is, \(\dom\trpi\) is invariant under conjugation and \(\trpi(gxg^{-1}) = \trpi(x)\) for all \(g\in\G\) and \(x\in\dom\trpi\).

In the following sections, we will show that \(\dom\trpi\) contains all regular semisimple elements, and given such an element~\(g\), we will describe a subgroup~\(\CG\) for which~\(\trpi\) is locally constant on \(\CG g\CG\).  We begin with some preparatory results.  First we describe the trace distribution as a limit of locally constant functions and relate the latter to the trace function.

Let~\(\CG\) be a compact open pro-\(p\) subgroup of~\(\G\) (these exist by \cite{Meyer-Solleveld:Resolutions}*{Lemma~1.1}).  Since the space~\(V^\CG\) of \(\CG\)\nb-invariants in~\(V\) is finite-dimensional, the linear operator \(\repr(\ide{\CG}g\ide{\CG})\) has finite rank for all \(g\in \G\).  Hence
\[
\chi_\CG(g) \defeq \tr(\repr(\ide{\CG}g\ide{\CG}),V)
= \mu(\CG)^{-1} \tr(\repr(1_{\CG g}),V)
= \mu(\CG)^{-1} \tr(\repr(1_{g\CG}),V)
\]
defines a \(\CG\)\nb-biinvariant function on~\(\G\); here we used that \(\repr(g\ide{\CG})\), \(\repr(\ide{\CG}g\ide{\CG})\), and \(\repr(\ide{\CG}g)\) have the same trace.  By construction,
\begin{equation}
  \label{eq:trace_as_integral}
  \tr(\repr(f),V) = \int_\G f(g) \chi_\CG(g)\,\diff\mu(g)
\end{equation}
for all \(\CG\)\nb-biinvariant compactly supported functions~\(f\) on~\(\G\).  Let \((\CG_n)_{n\in\N}\) be a decreasing sequence of compact open pro\nb-\(p\) subgroups with \(\bigcap \CG_n= \{1\}\).  Then any locally constant, compactly supported function is \(\CG_n\)\nb-biinvariant for some \(n\in\N\), so that~\eqref{eq:trace_as_integral} holds for \(\CG=\CG_n\) for all sufficiently large~\(n\).  In this sense, the trace distribution is the limit of the locally constant functions~\(\chi_\CG\) in a distributional sense.  The following lemma is trivial:

\begin{lemma}
  \label{lem:trace_exists}
  The trace function exists at \(\gamma\in\G\) and has value~\(\tau\) if and only if there is \(n_0\in\N\) with \(\chi_{\CG_n}(g) = \tau\) for all \(g\in \CG_{n_0}\gamma\CG_{n_0}\) and all \(n\ge n_0\).  Furthermore, then~\(\trpi\) is defined and constant on \(\CG_{n_0}\gamma\CG_{n_0}\).
\end{lemma}

Let \(\gamma\in\G\) be a regular semisimple element.  Then~\(\gamma\) is contained in some maximal torus~\(\T\).  Let \(\T^\reg\subseteq\T\) be the subset of regular elements.  It is well-known that the map
\begin{equation}
  \label{eq:psi}
  \psi\colon \G/\T \times \T^\reg \to \G,\qquad
  (gT,t) \mapsto g t g^{-1},
\end{equation}
is open.  We are going to quantify this statement by providing compact open subgroups \(\CG,\CG_\G\subseteq\G\), and \(\CG_\T\subseteq\T\) such that \(\psi(\CG_\G\T\times\CG_\T\gamma)\) contains \(\CG\gamma\CG\) for a given regular element~\(\gamma\) of~\(\T\).  We first consider the split case.

\begin{lemma}
  \label{lem:[gamma]}
  Suppose that~\(\T\) contains the maximal split torus~\(\ST\) of~\(\G\).  Then the map
  \[
  \Un^+ \to \Un^+\colon u \mapsto [u,\gamma]
  \]
  is a diffeomorphism.
\end{lemma}

\begin{proof}
For \(\alpha,\beta,\alpha + \beta \in \Phi\cup \{0\}\), we have \([\Un_\alpha , \Un_\beta] \subseteq
\Un_{\alpha + \beta}\), where we interpret~\(\Un_0\) as \(\Ce(\T)\).  Let~\(\Un^{(n)}\) be the
group generated by the~\(\Un_\alpha\) with \(\alpha \in \Phi^+\) of height at least~\(n\).  Then
  \[
  \Un^+ = \Un^{(1)} \supseteq \Un^{(2)} \supseteq \dotsb \supseteq
  \Un^{(\hgt(\Phi))} \supseteq \{1\}
  \]
  is a filtration of~\(\Un^+\) by normal subgroups.  Moreover, as algebraic groups
  \[
  \Un^{(n)} / \Un^{(n+1)}
  \cong \prod_{\alpha \in \Phi^{(n)}} \Un_\alpha / \Un_{2 \alpha}
  \]
  where~\(\Phi^{(n)}\) denotes the set of roots of height~\(n\).  The group \(\Un_\alpha / \Un_{2 \alpha}\) carries a canonical \(\F\)\nb-vector space structure, so we can speak of~\(\lambda u_\alpha\) for \(\lambda \in \F\) and \(u_\alpha \in \Un_\alpha / \Un_{2 \alpha}\).

  Given \(v \in \Un^+\), we recursively construct \(u_n\in\Un^{(n)}\) such that
  \[
  [u_n\dotsm u_2\cdot u_1,\gamma] \in v \Un^{(n+1)}.
  \]
  Then \(u\defeq u_{\hgt(\Phi)} \dotsm u_2 \cdot u_1\) belongs to~\(\Un^+\) and satisfies \([u,\gamma] = v\).  The construction will show that the~\(u_n\) and hence~\(u\) depend algebraically on~\(v\) and that the class of~\(u_n\) in \(\Un^{(n)}/\Un^{(n+1)}\) is unique.  It follows that the map \(u\mapsto [u,\gamma]\) is bijective and that the inverse map is algebraic.

  Let \(w_n \defeq [u_n\dotsm u_2\cdot u_1,\gamma]\) and define \(w_0\defeq 1\).  These elements satisfy the recursive relation
  \begin{multline*}
    w_n = u_n u_{n-1}\dotsm u_1 \gamma (u_{n-1}\dotsm u_1)^{-1} \gamma^{-1}\gamma u_n^{-1}\gamma^{-1}
    \\= w_{n-1} w_{n-1}^{-1} u_n w_{n-1} u_n^{-1} u_n \gamma u_n^{-1} \gamma^{-1}
    = w_{n-1} [w_{n-1}^{-1},u_n] [u_n,\gamma].
  \end{multline*}
  If \(u_n\in\Un^{(n)}\), then \([u_n,\gamma] \in \Un^{(n)}\) and \([w_{n-1}^{-1},u_n]\in \Un^{(n+1)}\) because \([\Un^+,\Un^{(n)}]\subseteq \Un^{(n+1)}\).  Since \(\Un^{(n)}/\Un^{(n+1)}\) is commutative, we have \(w_n \in w_{n-1} [u_n,\gamma] \Un^{(n+1)}\).  Hence~\(u_n\) must solve the equation \([u_n,\gamma] \in w_{n-1}^{-1}v \Un^{(n+1)}\).  As
  \[
  [u_\alpha ,\gamma] = \bigl(1 - \alpha(\gamma)\bigr) u_\alpha
  \qquad\text{for } u_\alpha \in \Un_\alpha / \Un_{2 \alpha},
  \]
  the map \([?,\gamma]\colon \Un^{(n)} / \Un^{(n+1)} \to \Un^{(n)} / \Un^{(n+1)}\) is invertible.  Since \(w_{n-1}^{-1}v\in \Un^{(n)}\) by induction assumption, there is a unique coset \(u_n\Un^{(n+1)}\) with \(w_{n-1}[u_n,\gamma] \Un^{(n+1)} = v\Un^{(n+1)}\), and it depends algebraically on \(w_{n-1}^{-1}v\).  We may pick a representative in this coset by an algebraic map.  If we do this in each step, then the final result~\(u\) depends algebraically on~\(v\) and satisfies \([u,\gamma]=v\).  In each step, there is a unique way of lifting a solution of the equation \([u,\gamma]=v\) from~\(\Un^+/\Un^{(n)}\) to~\(\Un^+/\Un^{(n+1)}\); in the first step, there is a unique solution in \(\Un^+/\Un^{(2)}\).  Hence there is a unique \(u\in\Un^+\) with \([u,\gamma]=v\).
\end{proof}

\begin{proposition}
  \label{prop:conjugatesplit}
  Suppose that the maximal torus~\(\T\) containing~\(\gamma\) is split, so that~\(\G\) is split.  Let~\(\Apa\) be the apartment corresponding to \(\ST=\T\), let \(x \in \Apa\), and let \(r\in\R_{\geq \sd(\gamma)} \).  Then the map~\(\psi\) in~\eqref{eq:psi} restricts to an injective map from \((\UC[x]^{(0)} / H_{0{+}}) \times H_{r{+}} \gamma\) onto a neighbourhood of~\(\gamma\) that contains \(\UC[x]^{(r)}\gamma\).
\end{proposition}

\begin{proof}
  First we prove injectivity on the indicated domain.  Assume \(\psi(g_1\T,t_1) = \psi(g_2\T,t_2)\).  Then \(g_2^{-1} g_1 t_1 g_1^{-1} g_2 = t_2 \in \T\).  Since~\(t_1\) is regular, this implies \(g_2^{-1} g_1 \in \No(\T)\).  But \(\No(\T) \cap \UC[x]^{(0)} = \Ce(\T) = \T\), so that \(g_1\T = g_2\T\) and therefore \(t_1=t_2\).

  Since~\(\Galg\) splits, the definition~\eqref{eq:Hsplit} yields \(H_{r{+}} \subseteq \T\).  As \(\psi (u ,h\gamma ) = [u, h\gamma ] h \gamma\), Lemma \ref{lem:[gamma]} shows that \(\psi (\G/\T \times H_{r{+}} \gamma)\) contains \(\Un^+ H_{r{+}} \gamma\) for any positive system \(\Phi^+ \subset \Phi\).  We may decompose any element of \(\UC[x]^{(r)} \gamma\) as \(y = y_+\cdot y_-\cdot y_0\) with \(y_\pm \in \Un^\pm \cap \UC[x]^{(r)}\) and \(y_0\in H_{r{+}} \gamma\).  There are \(u_+ \in \Un_+\) and \(u_- \in \Un_-\) such that
  \[
  y_+ y_0 = u_+ y_0 u_+^{-1} \quad\text{and}\quad
  y_- y_0 = u_- y_0 u_-^{-1}.
  \]
  Now \(\sd(y_0) = \sd(\gamma) \ge 0\) and \([u_+, y_0] = y_+ \in \UC[x]^{(r)}\) force \(u_+ \in \UC[x]^{(r - \sd(\gamma))} \subseteq \UC[x]^{(0)}\).  For the same reason, \(u_- \in \UC[x]^{(r - \sd(\gamma))}\).  A good approximation for \(\psi^{-1}(y)\) is \((u_- u_+, y_0)\):
  \begin{multline}
    \label{eq:psiu}
    \psi (u_- u_+, y_0) = u_- u_+ y_0 u_+^{-1} u_-^{-1}
    = u_- y_+ y_0 u_-^{-1}
    \\= u_- y_+ u_-^{-1} y_- y_0
    = [u_- ,y_+] y_+ y_- y_0
    = [u_- ,y_+] y.
  \end{multline}
  Theorem~\ref{thm:UOmegae}.\ref{UOmegae_8} yields
  \[
  [u_- ,y_+] \in [\UC[x]^{(r - \sd(\gamma))}, \UC[x]^{(r)}]
  \subseteq \UC[x]^{(2r - \sd(\gamma))} ,
  \]
  but we can be more precise.  Let \(r' > r\) the smallest number with \(\UC[x]^{(r')} \neq \UC[x]^{(r)}\).  Choose \(\epsilon \in (0, r'-r)\) such that \(\UC[x]^{(\epsilon)} = \UC[x]^{(0)}\) (this is possible because the filtrations \eqref{eq:Ualphar} and~\eqref{eq:Hr} are discrete).  Now Theorem~\ref{thm:UOmegae}.\ref{UOmegae_8} yields
  \[
  [u_- ,y_+] \in \UC[x]^{(r')}.
  \]
  In other words, \(\psi (u_- u_+, y_0) = y\) in \(\PC[x]/ \UC[x]^{(r')}\).

  Next we try to find a solution of the form \(\psi (u_- u_+ g, t y_0) = y\).  By~\eqref{eq:psiu} this is equivalent to
  \[
  \psi (g, t y_0) = u_+^{-1} u_-^{-1} y u_- u_+
  =  (u_- u_+)^{-1} [y_+,u_-] (u_- u_+) \, y_0.
  \]
  Since \(u_- u_+ \in \UC[x]^{(0)} \subseteq \No\bigl( \UC[x]^{(r')} \bigr)\), the right hand side lies in \(\UC[x]^{(r')} y_0\).  Thus we transformed the original problem
  \[
  \psi \bigl( \UC[x]^{(0)} / H_{0+} \times H_{r{+}} \gamma \bigr) \supseteq \UC[x]^{(r)} \gamma
  \]
  to the problem
  \[
  \psi \bigl( \UC[x]^{(0)} / H_{0+} \times H_{r{+}} y_0 \bigr) \supseteq \UC[x]^{(r')} y_0.
  \]
  Since \(H_{r{+}}\gamma = H_{r{+}} y_0\), \(r'>r\) and
  \[
  \UC[x]^{(r')} y_0 \subseteq \UC[x]^{(r')} H_{r{+}} \gamma \subsetneq \UC[x]^{(r)} \gamma,
  \]
  repetition of this process yields a solution \(\psi^{-1}(y)\).
\end{proof}

Now we consider a regular element~\(\gamma\) of a non-split maximal torus \(\T = \Talg (\F)\).  Furthermore, we want to generalise the statement by allowing the choice of an arbitrary \(x\in \Apa\).  Let~\(\tilde{\F}\) be a splitting field of~\(\Talg\), let \(\tilde{\G}=\Galg(\tilde{\F})\), and let \(\tilde{\T} \defeq \Talg(\tilde{\F})\).  This is a split maximal torus in~\(\tilde{\G}\), which therefore corresponds to an apartment~\(\tApa{\tilde{\T}}\) in the building \(\bt{\tilde{\F}}\).  Recall the subgroups \(\tilde{H}_r \subseteq \Calg[\Galg(\tilde{\F})]\bigl(\Talg(\tilde{\F})\bigr)\).

For \(x \in \bt{\F}\), let~\(\piTx\) be the point of~\(\tApa{\tilde{\T}}\) that is nearest to~\(x\).  Let~\(\Psi\) be the root system corresponding to an apartment of \(\bt{\tilde{\F}}\) that contains \(x\) and~\(\piTx\).  We define
\begin{equation}
  \label{eq:dTx}
  d_\T (x) \defeq \max_{\beta \in \Psi^\red}  \abs{\beta(\piTx) - \beta(x)}.
\end{equation}
If~\(\tilde{\F}/\F\) is tamely ramified, then~\eqref{eq:Rousseau} shows that \(\tApa{\tilde{\T}} \cap \bt{\F}\) is non-empty, that is, there is~\(x\) with \(d_\T(x)=0\).

Alternatively, let \(\tilde{C}\subseteq \tApa{\tilde{\T}} \) be a chamber containing~\(\piTx\), let \(\rho_{\tApa{\tilde{\T}} ,\tilde{C}}\colon \bt{\tilde{\F}} \to \tApa{\tilde{\T}} \) be the associated retraction.  Then
\[
d_\T (x) = \max_{\alpha \in \tilde{\Phi}^\red}
\abs{\alpha(\piTx) - \alpha(\rho_{\tApa{\tilde{\T}} ,\tilde{C}}(x))}.
\]
Proposition~\ref{prop:fixpoints}.(c) yields
\begin{equation}\label{eq:dTxgamma}
d_\T (x) \leq \hgt(\Phi) \sd(\gamma)\qquad \text{for all } x \in \bt{\tilde \F}^\gamma .
\end{equation}
Lemma~\ref{lem:Uetildea} and~\eqref{eq:f*Omega} yield
\begin{equation}
  \label{eq:UdTx}
  \UC[x]^{(r + d_\T(x))} = \tUC[x]^{(r + d_\T(x))} \cap \Galg(\F)
  \subseteq \tUC[\piTx]^{(r)} \cap \UC[x].
\end{equation}

\begin{lemma}
  \label{lem:conjugatedTx}
  Let \(\gamma \in \T\) be regular and let \(r\in\R_{\ge \sd(\gamma)}\).  Let \(x \in \bt{\F}\) and abbreviate \(\Kx = \tUC[\piTx]^{(0)}\cap\G\).  Then \(\UC[x]^{(r+d_\T(x))} \gamma\) is contained in \(\psi\bigl( \Kx \times (\tilde{H}_{r{+}} \gamma \cap \T)\bigr)\).
\end{lemma}

\begin{proof}
  Equation~\eqref{eq:UdTx} and Proposition~\ref{prop:conjugatesplit} show that every element of \(\tUC[x]^{(r+d_\T(x))} \gamma\) is conjugate in \(\Galg(\tilde{\F})\) to an element of \(\tilde{H}_{r{+}} \gamma \cap \Talg(\tilde{\F})\).  Since the maps
  \[
  \tilde{\psi}\colon \bigl(\Galg (\tilde{\F}) / \Talg(\tilde{\F})\bigr)
  \times \Talg(\tilde{\F}) \to \Galg(\tilde{\F})
  \quad \text{and} \quad
  \psi\colon \bigl(\Galg (\F) / \Talg(\F)\bigr) \times \Talg(\F) \to \Galg(\F)
  \]
  are injective and open, respectively on \(\tUC[\piTx]^{(0)}/\tilde{H}_{0{+}} \times \bigl(\tilde{H}_{r{+}} \gamma \cap \Talg(\tilde{\F}\bigr))\) and on the intersection of this set with~\(\G\),
  \[
  \tilde{\psi} \bigl( \Kx \times (\tilde{H}_{r{+}} \gamma \cap \T) \bigr) =
  \tilde{\psi} \bigl( \tilde{\Un}_{\piTx}^{(0)}
  \times (\tilde{H}_{r{+}} \gamma \cap \Talg(\tilde{\F})) \bigr) \cap \G.
  \]
  Moreover, by Proposition~\ref{prop:conjugatesplit} the right hand side contains
  \begin{equation}
    \tilde{\Un}_{\piTx}^{(r)} \gamma \cap \G \supseteq
    \tUC[x]^{(r + d_\T (x))} \gamma \cap \G = \UC[x]^{(r + d_\T (x))} \gamma.\qedhere
  \end{equation}
\end{proof}

There is a decreasing sequence \((\CG_n)_{n\in\N}\) of \emph{normal} compact open subgroups in~\(\Kx\) with \(\bigcap \CG_n = \{1\}\).  Since~\(\Kx\) is open in~\(\G\), we may use this sequence to approximate the trace distribution as in~\eqref{eq:trace_as_integral}.  Since~\(\CG_n\) is normal in~\(\Kx\), then the space of \(\CG_n\)\nb-biinvariant functions is invariant under conjugation by elements of~\(\Kx\).  This implies that the function~\(\chi_{\CG_n}\) is invariant under conjugation by elements of~\(\Kx\).  Therefore, Lemma~\ref{lem:conjugatedTx} shows that~\(\chi_{\CG_n}\) is constant on \(\UC[x]^{(r+d_\T(x))}\gamma\) once it is constant on \(\tilde{H}_{r{+}}\gamma\cap\T\).  In the following, we may therefore restrict attention to elements of a torus in~\(\G\).

\section{The local constancy of characters}
\label{sec:constancycpt}

Let \((\repr,V)\) be an admissible representation of~\(\G\) in good characteristic, of level \(e \in \Z_{\ge 0}\).  Let~\(\gamma\) be a regular semisimple element of a maximal torus \(\T \subseteq \G\) and let \(x\in\bt{\F}^\circ\) be a vertex in the building of~\(\G\).  We are going to find \(\rmax\in\N\) depending only on~\(\gamma\) and the level~\(e\) of the representation, such that \(\trpi\) is defined and constant on \(\UC[x]^{(\rmax+d_\T(x))}\) with \(d_\T(x)\) as in~\eqref{eq:dTx}.

\subsection{Local constancy for compact elements}
First we assume, in addition, that~\(\gamma\) is a compact element, so that~\(\gamma\) fixes some point in the affine building.  The assertions for general elements are reduced to the compact case in Section~\ref{sec:constancygen}.

Our definition of~\(\rmax\) is somewhat complicated and probably not optimal.  It is likely that \(\rmax = \max\{\sd(\gamma),e\}\) works, but we can only prove this if~\(\T\) has a subtorus~\(\ST\) that is a maximal \(\F\)\nb-split torus of~\(\G\).

Let \(\T = \Talg (\F) \subseteq \G\) be a maximal torus containing~\(\gamma\) and let~\(\tilde{\F} / \F\)
be a finite Galois extension splitting \(\Talg\).  Recall the subgroups 
\(\tilde{\Un}^+ \subset \Galg (\tilde{\F})\) and 
\(\tilde{H}_r \subseteq \Calg[\Galg(\tilde{\F})]\bigl(\Talg(\tilde{\F})\bigr)\).  
Let~\(\tilde{B}\) be a Borel subgroup of \(\Galg (\tilde{\F})\) containing \(\Talg(\tilde{\F})\)
and let $\tilde{\Un}^+_T$ be its unipotent radical. We define $\Un^+_T = \tilde{\Un}^+_T \cap G$.
Since $\tilde B$ need not be defined over $\F$, $T \ltimes \Un^+_T$ need not be a Borel subgroup
of $G$. In fact $\Un^+_T$ can be reduced to $\{1\}$, for example if $T$ is anisotropic.

\begin{definition}
  \label{def:rmax}
  For \(x\in\bt{\F}\) define \(d_\T (x)\) as in~\eqref{eq:dTx} and let \(d(\gamma) \in \R\) be the smallest number such that
  \begin{equation}
    \label{eq:dgamma}
    \bt{\F}^\gamma \subseteq \Un^+_T \cdot \{ x \in \bt{\F} : d_\T (x) \le d(\gamma)\}.
  \end{equation}
\end{definition}

We have \(d(\gamma)<\infty\) because \(\bt{\F}^\gamma/\T\) is compact.

\begin{theorem}
  \label{thm:localconst}
  Define \(\rmax \defeq \max \{ \hgt(\Phi) \sd(\gamma), e+d(\gamma)\}\).
  \begin{enumerate}[label=\textup{(\alph{*})}]
  \item The function \(\trpi\) is defined and constant on \(\gamma \tilde{H}_{\rmax+} \cap \T\), and on all \(\G\)-conjugacy classes intersecting this set.
  \item The function \(\trpi\) is constant on \(\UC[x]^{(\rmax + d_\T (x))}\gamma\), for any \(x \in \bt{\F}\).
  \item If~\(\T\) has a subtorus~\(\ST\) that is a maximal \(\F\)\nb-split torus of~\(\G\), then \(d(\gamma) =0\) and we may omit the factor \(\hgt (\Phi)\) in the definition of~\(\rmax\), that is, \(\trpi\) is constant on \(\gamma \tilde{H}_{\max \{ \sd(\gamma), e\}+ } \cap \T\).
  \end{enumerate}
\end{theorem}

If \(\tilde{\F} / \F\) is tamely ramified, then~\eqref{eq:Rousseau} shows that there is a point \(x \in \bt{\F}\) with \(d_\T (x) = 0\), so that \(\trpi\) is constant on \(\UC[x]^{(\rmax)}\gamma\).

The number~\(\rmax\) will reappear frequently in the following.  We will not need the definition of~\(\rmax\) but only Theorem~\ref{thm:localconst}.(a).  That is, the following results remain true for a smaller value of~\(\rmax\) provided Theorem~\ref{thm:localconst}.(a) can be established for it.

\begin{proof}
(a) Theorem~\ref{thm:resolution} implies a formula for \(\tr(\repr(f),V)\), which is worked out in 
\cite{Meyer-Solleveld:Resolutions}*{Proposition~4.1}.  We need some notation to state this trace formula.  
For \(g\in\G\), let~\(\Sigma^g\) be the set of all polysimplices~\(\sigma \in \Sigma\) with 
\(g\sigma=\sigma\) and let \(\epsilon_\sigma(g)=\pm1\), depending on whether the automorphism 
of~\(\sigma\) induced by~\(g\) preserves or reverses orientation.  For a locally constant 
function~\(f\) supported in~\(\PC[x]\), \cite{Meyer-Solleveld:Resolutions}*{Proposition~4.1} asserts
  \begin{equation}
    \label{eq:traceformula}
    \tr(\repr(f),V) = \lim_\Sigma \int_{g\in \PC[x]}
    f(g) \sum_{\sigma\in\Sigma^g} (-1)^{\dim \sigma} \epsilon_\sigma(g)
    \tr\bigl(\repr(g),V^{\UC[\sigma]^{(e)}} \bigr) \,\diff\mu(g),
  \end{equation}
  where the limit means that there is a finite convex subcomplex~\(\Sigma_0\) such that the right 
  hand side is the same for all \(\PC[x]\)\nb-invariant finite convex subcomplexes~\(\Sigma\) of 
  \(\bt{\F}\) with \(\Sigma\supseteq\Sigma_0\).  Thus we want to show that the function
  \begin{equation}
    \label{eq:tauSigma}
    \tau_\Sigma\colon g\mapsto
    \sum_{\sigma\in\Sigma^g} (-1)^{\dim \sigma} \epsilon_\sigma(g)
    \tr\bigl(\repr(g),V^{\UC[\sigma]^{(e)}} \bigr)
  \end{equation}
  is constant on \(\UC[x]^{(\rmax + d_\T (x))}\gamma\) for all sufficiently large \(\PC[x]\)\nb-invariant finite convex subcomplexes~\(\Sigma\).  The function~\(\tau_\Sigma\) is invariant under conjugation by elements of~\(\PC[x]\) because~\(\Sigma\) is \(\PC[x]\)\nb-invariant.

  Lemma~\ref{lem:constantfix} yields \(\bt{\F}^g = \bt{\F}^\gamma\) for all \(g \in \tilde{H}_{\rmax{+}} \gamma \cap \T\), because \(\rmax \ge \hgt(\Phi) \sd(\gamma)\).  Since
  \begin{equation}
    \label{eq:HedTx}
    \tilde{H}_{e + d_\T (x){+}} \subseteq \tUC[\piTx]^{(e+d_\T (x))} \subseteq \tUC[x]^{(e)} ,
  \end{equation}
  the operator \(\pi (g^{-1} \gamma)\) restricts to the identity on \(V^{\UC[x]^{(e)}}\), for all~\(x\) with \(d_\T (x) \le d (\gamma )\).

Let~\(\mathcal D\) be a set of simplices in \(\bt{\F}^\gamma\), such that \(\mathcal D\) is a 
fundamental domain for the action of \(\Un^+_T\) on \(\Un^+_T \cdot \bt{\F}^\gamma\) and every 
\(\sigma \in \mathcal D\) contains an interior point~\(x\) with \(d_\T (x) \le d(\gamma)\). 
Equation~\eqref{eq:tauSigma} becomes
  \begin{equation}
    \label{eq:sumtheta}
    \tau_\Sigma (g) = \sum_{u \sigma \in \Sigma^g} \epsilon_{u \sigma}(g)
    \tr \bigl( \pi (g) , V^{\UC[u \sigma]^{(e)}} \bigr) =
    \sum_{u \sigma \in \Sigma^g} \epsilon_{\sigma}(u^{-1}g u)
    \tr \bigl( \pi (u^{-1}gu) , V^{\UC[\sigma]^{(e)}} \bigr) ,
  \end{equation}
  where the sum runs over all polysimplices \(u \sigma \in \Sigma^g = \Sigma^\gamma\) with 
  \(\sigma \in \mathcal D\) and \(u \in \Un^+_T\).  
% We want to show that \(\tau_\Sigma (\gamma) = \tau_\Sigma (g)\).
  Notice that we pick only one~\(u\) 
  for each such polysimplex.  Given another \(u_1 \in \Un^+_T\) with \(u_1 \sigma = u\sigma\), 
  we have \(u_1^{-1} u \in \PC[\sigma]\), so \(\theta (u_1 ,g) = \theta (u,g)\), where
  \[
  \theta (u,g) \defeq \epsilon_{\sigma}(u^{-1}g u)
  \tr \bigl( \pi (u^{-1}gu) , V^{\UC[\sigma]^{(e)}} \bigr).
  \]
  Recall that 
  \[
  \{ u \in \tilde{\Un}^+_T : [u^{-1},g] \in \tilde{\PC[\sigma]} \} =
  \{ u \in \tilde{\Un}^+_T : g (u \sigma) = u \sigma \} .
  \]
  The equality $\Sigma^g = \Sigma^\gamma$ implies that
  \[
  \{ u \in \tilde{\Un}^+ : [u^{-1},g] \in \tilde{\PC[\sigma]} \} =
  \{ u \in \tilde{\Un}^+ : [u^{-1},\gamma] \in \tilde{\PC[\sigma]} \} .
  \]
  We denote this set by $\tilde{\Un}^+_T (\sigma)$ and we write $\Un^+_T (\sigma) =
  \tilde{\Un}^+_T (\sigma) \cap G$. Let $\mu$ be a Haar measure on $\Un^+_T$, normalised
  so that $\mu (\Un^+_T (\sigma)) = 1$.
  Now we can rewrite \eqref{eq:sumtheta} as
  $\tau_\Sigma (g) = \sum_{\sigma \in \mathcal D} \tau_{\Sigma,\sigma}(g)$, where
  \begin{equation}\label{eq:tausigma}
  \tau_{\Sigma,\sigma}(g) = \sum_{u \sigma \in \Sigma^g \cap \Un^+_T \sigma}
  \theta (u,g) = \sum_{u \in \Un^+_T (\sigma) / \Un^+_T \cap \PC[\sigma]} \theta (u,g) =
  \int_{\Un^+_T (\sigma)} \theta (u,g) \textup{d}\mu (u) .
  \end{equation}
  By Lemma~\ref{lem:[gamma]} the map 
  \[
  \phi_g : \tilde{\Un}^+ \to \tilde{\Un}^+\colon u \mapsto [u^{-1},g] 
  \]
  is a diffeomorphism. We claim that $\phi_g^{-1} (\Un^+_T) = \Un^+_T$.
  
  It is clear that $\phi_g^{-1} (\Un^+_T) \supset \Un^+_T$. Suppose that $u \in
  \phi_g^{-1} (\Un^+_T)$, that is, $u^{-1} g u g^{-1} \in \tilde{\Un}^+_T \cap G$.
  For any $\tau \in \mathrm{Gal} (\tilde{\F} / \F)$ we have 
  \[
  G \ni u^{-1} g u = \tau (u^{-1} g u) = \tau (u^{-1}) g \tau (u) , 
  \]
  so $\tau (u) u^{-1} \in Z_{\Galg (\tilde{\F})}(g) = \Talg (\tilde{\F})$. Since
  $\tau (u)$ is unipotent and $\tau (u) \in \Talg (\tilde{\F}) u \subset \tilde{B}$,
  we must have $\tau (u) \in \tilde{\Un}^+_T$. Then $\tau(u) u^{-1} \in
  \tilde{\Un}^+_T \cap \Talg (\tilde{\F}) = \{1\}$, which means that $\tau (u) = u$.
  Hence $u \in \tilde{\Un}^+_T \cap \Galg (\tilde{\F})^{\mathrm{Gal}(\tilde{\F} / \F)}
  = \Un^+_T$.
  
  As $\tilde{\Un}^+_T (\sigma) = \phi_g^{-1}  (\tilde{\Un}^+_T (\sigma) \cap 
  \tilde{\PC[\sigma]} )$, $\phi_g$ restricts to a diffeomorphism
  \[
  \phi_g : \Un^+_T (\sigma) \to \PC[\sigma] \cap \Un^+_T.
  \]
  The same statements hold for $\phi_\gamma : \tilde{\Un}^+_T \to \tilde{\Un}^+_T$.
  For every $u_1 \in \Un^+_T (\sigma)$, thought of as appearing in $\tau_{\Sigma,\sigma}(g)$,
  the element $u_2 = \phi_\gamma^{-1} (\phi_g (u_1))$ appears in 
  $\tau_{\Sigma_,\sigma} (\gamma)$. We do not know whether $u_1 \sigma$ equals $u_2 \sigma$
  or not, fortunately that is not needed. Now
  \begin{align*}
    \theta (u_2 ,\gamma) & = \epsilon_{\sigma}\bigl( [u_2^{-1} ,\gamma] \gamma \bigr)
    \tr \bigl( \pi ([u_2^{-1} ,\gamma] \gamma) , V^{\UC[\sigma]^{(e)}} \bigr) \\
    & = \epsilon_{\sigma}\bigl( [u_1^{-1} ,g] g (g^{-1}\gamma) \bigr)
    \tr \bigl( \pi ([u_1^{-1} ,g]) \pi (g) \pi (g^{-1} \gamma) , V^{\UC[\sigma]^{(e)}} \bigr).
  \end{align*}
  Since \(\Sigma^g = \Sigma^\gamma\), \(g^{-1}\gamma\) fixes~\(\sigma\) pointwise, while in 
  view of~\eqref{eq:HedTx} and the definition of~\(\mathcal D\), \(\pi (g^{-1}\gamma)\) 
  acts as the identity on~\(V^{\UC[\sigma]^{(e)}}\).  Therefore
  \begin{equation}\label{eq:equaltheta}
  \theta (u_2 ,\gamma ) = \epsilon_{\sigma}\bigl( [u_1^{-1} ,g] g) \bigr)
  \tr \bigl( \pi ([u_1^{-1} ,g] g) , V^{\UC[\sigma]^{(e)}} \bigr) = \theta (u_1 ,g),
  \end{equation}
  which shows that every term of the integral \eqref{eq:sumtheta} also occurs in 
  \(\tau_{\Sigma,\sigma} (\gamma)\).  
  
  Like in the proof of Lemma~\ref{lem:[gamma]}, the generalised eigenvalues of the differentials 
  \[
  D\phi_\gamma , D \phi_g\colon \mathrm{Lie}_{\tilde{\F}} (\Un^+_T) \to \mathrm{Lie}_{\tilde{\F}} (\Un^+_T) 
  \]
  are \(\{ 1 - \alpha (\gamma) : \alpha \in \tilde{\Phi}^+ \}\) and \(\{ 1 - \alpha (g) : 
  \alpha \in \tilde{\Phi}^+ \}\), and they occur with multiplicity \(d_\alpha \defeq \dim 
  \mathrm{Lie}_{\tilde{\F}} (\Ualg_\alpha / \Ualg_{2 \alpha})\).  
  The restriction \(h = \gamma^{-1} g \in \tilde H_{\sd (\gamma)+} \cap \Talg (\tilde{\F}) \) implies
  \[
  v (1 - \alpha (g)) = v \bigl(1 - \alpha (\gamma) \alpha (h) \bigr) = v \bigl( 1 - \alpha (\gamma) +
  \alpha (\gamma) (1 - \alpha (h)) \bigr) = v \bigl( 1 - \alpha (\gamma) \bigr)
  \]
  for all \(\alpha \in \tilde{\Phi}\). Hence $\phi_g, \phi_\gamma : \tilde{\Un}^+_T \to \tilde{\Un}^+_T$
  both have Jacobian $\prod_{\alpha \in \tilde{\Phi}^+} \norm{1 - \alpha (g)}_{\tilde{\F}}^{d_\alpha}$.
  It follows that $\phi_g, \phi_\gamma : \Un^+_T \to \Un^+_T$ both have Jacobian
  $\prod_{\alpha} \norm{1 - \alpha (g)}_{\tilde{\F}}^{d_\alpha / [\tilde{\F} : \F]}$, where
  the product runs over all $\alpha \in \tilde{\Phi}^+$ such that the restriction of $\alpha$ to
  $T$ appears in $\Lie (\Un^+_T)$.
  As the diffeomorphisms $\phi_g ,\phi_\gamma : \Un^+_T (\sigma) \to \PC[\sigma] \cap \Un^+_T$ have
  the same Jacobian, and by \eqref{eq:equaltheta},
  \begin{multline*}
  \tau_{\Sigma,\sigma}(g) = \int_{\Un^+_T (\sigma)} \theta (u,g) \textup{d}\mu (u) =
  \int_{\Un^+_T (\sigma)} \theta (\phi_\gamma^{-1} (\phi_g (u)) ,\gamma) \textup{d}\mu (u) \\
  = \int_{\Un^+_T (\sigma)} \theta (u,\gamma) \textup{d}\mu (u) = \tau_{\Sigma,\sigma}(\gamma) . 
  \end{multline*}
  This holds for all $\sigma \in \mathcal D$, so 
  \[
  \tau_\Sigma (g) = \sum_{\sigma \in \mathcal D} \tau_{\Sigma,\sigma}(g) =
  \sum_{\sigma \in \mathcal D} \tau_{\Sigma,\sigma}(\gamma) = \tau_\Sigma (\gamma) .
  \]
(b) Lemma~\ref{lem:conjugatedTx} shows that any element of \(\UC[x]^{(\rmax + d_\T (x))}\gamma\) is
\(\PC[x]\)\nb-conjugate to one of \(\gamma \tilde{H}_{\rmax{+}} \cap \T\).  Hence (b) follows from (a).

(c) To a large extent we will copy the proof of part (a), but we take advantage of 
\(\Un^+ \cdot \Apa = \bt{\F}\).  This clearly implies \(d(\gamma) = 0\), so that~\(\mathcal D\) 
is a collection of simplices of~\(\Apa\). 
%that form a fundamental domain for the action of \(\Ce(\ST)\) on~\(\Apa\).  
This~\(\mathcal D\) works for both \(\gamma\) and \(g = \gamma h\). We may replace 
$\hgt(\Phi) \sd(\gamma)$ by $\sd (\gamma)$ in the definition on $r(\gamma)$, because the
factor $\hgt(\Phi)$ was only needed to ensure that \(\bt{\F}^g\) equals \(\bt{\F}^\gamma\).

With these choices the proof of (a) mostly goes through, even though we do not know whether 
or not $g$ and $\gamma$ have the same fixed points. We now have two possibly different sets
$\Un^+_T (\sigma)$, one for $g$ and one for $\gamma$, but the map $\phi_\gamma^{-1} \circ
\phi_g$ still provides a diffeomorphism between them.
\end{proof}

\subsection{Local constancy for non-compact elements}
\label{sec:constancygen}

We would like to generalise Theorem~\ref{thm:localconst} to all regular semisimple elements.  This is possible using Jacquet modules and parabolic restriction as in~\cite{Casselman:Characters_Jacquet}.  Although the methods in~\cite{Casselman:Characters_Jacquet} are algebraic and not restricted to complex coefficients, Casselman refers to earlier work which was written with complex representations in mind.  This makes it hard to judge whether Casselman's proofs work for representations in good characteristic.  Fortunately, Vign\'eras~\cite{Vigneras:l-modulaires} proved the required results in this generality.

Let \(\gamma\in\T\) be a semisimple element and let \(\Pa_\gamma\subseteq\G\) be the parabolic subgroup contracted by~\(\gamma\), which is defined in~\eqref{eq:Pg}.  Since~\(\F\) is complete with respect to the valuation~\(v\), Proposition~\ref{prop:Pg}.\ref{prop:Pgd} shows that~\(\gamma\) is compact in~\(\Le_\gamma\).  It follows from Proposition \ref{prop:Pg}.\ref{prop:Pgb} that \(\Lie\bigl(\Ralg[\Palg_\gamma]\bigr) \subseteq \Lie(\Galg)\) is the sum of all eigenspaces of \(\Ad(\gamma)\) corresponding to eigenvalues with strictly positive valuation.  (Although the eigenvalues may lie in a field extension of~\(\F\), this subspace is defined over~\(\F\).)  Similarly, \(\Ru[\Pa_{\gamma^{-1}}]\) corresponds to the \(\gamma\)\nb-eigenvalues with strictly negative valuation.

The description of (standard) parabolic subgroups in Definition~\ref{def:standard_parabolic} shows that~\(\Le_\gamma\) contains a maximal split torus of~\(\G\), say~\(\ST_\gamma\).  It may happen that \(\gamma\notin\ST_\gamma\).  Let~\(x\) be a point of the apartment~\(\Apa[\gamma] \) of \(\bt{\F}\) corresponding to~\(\ST_\gamma\).  Proposition~\ref{prop:udp} implies
\begin{equation}
  \label{eq:wellplaced}
  \UC[x]^{(e)}
  = \bigl( \UC[x]^{(e)} \cap \Ru[\Pa_{\gamma^{-1}}] \bigr)
  \bigl( \UC[x]^{(e)} \cap \Le_\gamma \bigr)
  \bigl( \UC[x]^{(e)} \cap \Ru[\Pa_\gamma] \bigr),
\end{equation}
or, in other words, \(\UC[x]^{(e)}\) is well-placed with respect to \((\Pa_\gamma, \Le_\gamma)\).  The collection
\[
X = \{ g x  \in \bt{\F} :
\text{\(g\) lies in the maximal compact subgroup of \(\T\)}\}
\]
is finite and \(\gamma\)\nb-invariant.  Since \(\T\subset\Le_\gamma\), the subgroup \(\UC[x']^{(e)}\) is well-placed with respect to \((\Pa_\gamma, \Le_\gamma)\) for every \(x'\in X\).  The group \(\CG^{(e)} \defeq \bigcap_{x'\in X} \UC[x']^{(e)}\) is also well-placed:
\[
\CG^{(e)} =
\bigl( \CG^{(e)} \cap \Ru[\Pa_{\gamma^{-1}}] \bigr)
\bigl( \CG^{(e)} \cap \Le_\gamma \bigr)
\bigl( \CG^{(e)} \cap \Ru[\Pa_\gamma] \bigr) \eqdef
\CG^{(e)}_- \CG^{(e)}_0 \CG^{(e)}_+.
\]
It follows that
\[
\gamma \CG^{(e)}_- \gamma^{-1} \supsetneq \CG^{(e)}_-,\qquad
\gamma \CG^{(e)}_0 \gamma^{-1} = \CG^{(e)}_0,\qquad
\gamma \CG^{(e)}_+ \gamma^{-1} \subsetneq \CG^{(e)}_+ ,
\]
so that the sequence \(\CG^{(e)}\) for \(e\in\N\) has all the properties claimed in~\cite{Deligne:Support}.

\begin{theorem}[\cite{Vigneras:l-modulaires}*{II.3.7}]
  \label{thm:subspaces}
  Let \((\repr,V)\) be an admissible smooth \(\G\)\nb-representation in good characteristic and let \(g\in\G\) be such that \(\Pa_g=\Pa_\gamma\).  There exist increasing sequences of finite-dimensional vector spaces \(V^{(e)} \subseteq V^{\CG^{(e)}}\) and \(V_{\Ru[\Pa_\gamma]}^{(e)} \subseteq V_{\Ru[\Pa_\gamma]}^{\CG_0^{(e)}}\) such that
  \begin{enumerate}[label=\textup{(\alph{*})}]
  \item \(\bigcup_e V^{(e)} \oplus V(\Ru[\Pa_\gamma]) = V\) and \(\bigcup_e V_{\Ru[\Pa_\gamma]}^{(e)} = V_{\Ru[\Pa_\gamma]}\),
  \item The quotient map \(V \to V / V(\Ru[\Pa_\gamma]) = V_{\Ru[\Pa_\gamma]}\) restricts to bijections \(V^{(e)} \to V_{\Ru[\Pa_\gamma]}^{(e)}\) and \(\bigl( \bigcup_r V^{(r)} \bigr)^{\CG^{(e)}} \to V_{\Ru[\Pa_\gamma]}^{\CG_0^{(e)}}\),
  \item \(V^{(e)}\) is stable under \(\repr (1_{\CG^{(e)}gK^{(e)}})\).
  \end{enumerate}
\end{theorem}

This setup allows us to use the (elementary) arguments from \cite{Casselman:Characters_Jacquet}*{page 104}, which result in
\begin{equation}
  \label{eq:Cas}
  \tr\bigl(\mu(\CG^{(e)}gK^{(e)})^{-1} \repr(1_{\CG^{(e)} g \CG^{(e)}}),
  V\bigr) =
  \tr\bigl( \repr_{\Ru[\Pa_\gamma]}(g),
  V_{\Ru[\Pa_\gamma]}^{\CG_0^{(e)}} \bigr)
\end{equation}
for all \(g\in\G\) with \(\Pa_g=\Pa_\gamma\).  Notice that the set of such~\(g\) is contained in~\(\Le_\gamma\), so it is not open in~\(\G\) unless~\(\gamma\) is compact in~\(\G\).

\begin{theorem}
  \label{thm:trJacquet}
  Let~\(\gamma\) be a regular semisimple element.  Then \(\trpi(\gamma)\) and \(\tr_{\repr_{\Ru[\Pa_\gamma]}}(\gamma)\) are both defined, and they are equal.
\end{theorem}

\begin{proof}
  Since~\(\gamma\) is compact in~\(\Le_\gamma\), Theorem~\ref{thm:localconst} tells us that \(\tr_{\repr_{\Ru[\Pa_\gamma]}}\) is well-defined and constant near~\(\gamma\).  Pick an \(e\in\N\) such that it is constant on \(\gamma \CG_0^{(e)}\).  Now~\eqref{eq:Cas} yields
  \begin{align*}
    \tr_{\repr_{\Ru[\Pa_\gamma]}} (\gamma) & = \tr \bigl( \repr_{\Ru[\Pa_\gamma]}
    \bigl(\gamma*\ide{\CG_0^{(e)}} \bigr), V_{\Ru[\Pa_\gamma]}\bigr)
    = \tr \bigl( \repr_{\Ru[\Pa_\gamma]}(\gamma), V_{\Ru[\Pa_\gamma]}^{\CG_0^{(e)}} \bigr) \\
    & = \tr \bigl( \mu( \CG^{(e)} \gamma \CG^{(e)})^{-1} \repr(1_{\CG^{(e)} g \CG^{(e)}}) , V \bigr).
  \end{align*}
  As the subsets \(\CG^{(e)} \gamma \CG^{(e)}\) form a neighbourhood basis of~\(\gamma\) in~\(\G\), taking the limit \(e \to \infty\) and invoking Lemma~\ref{lem:trace_exists} shows that \(\trpi (\gamma)\) is well-defined and equals \(\tr_{\repr_{\Ru[\Pa_\gamma]}} (\gamma)\).
\end{proof}

This theorem, which Casselman~\cite{Casselman:Characters_Jacquet} proved for complex representations, enables us to reduce the computation of traces from general semisimple elements to compact semisimple elements.  Theorem~\ref{thm:localconst} tells us on which neighbourhood of~\(\gamma\) the function \(\tr_{\repr_{\Ru[\Pa_\gamma]}}\) is constant.  But this is only a neighbourhood in~\(\Le_\gamma\).  We also want to know on which neighbourhood in~\(\G\) the function~\(\trpi\) is constant.  Let~\(\rmax\) be such that Theorem~\ref{thm:localconst}.(a) holds when we view~\(\gamma\) as a compact element in~\(\Le_\gamma\).

\begin{theorem}
  \label{thm:trconst}
  Let~\(\gamma\) be a regular element of a \textup(not necessarily split\textup) maximal torus~\(\T\) of~\(\G\).  Let \((\repr,V)\) be an admissible representation of~\(\G\) of level~\(e\) in good characteristic.
  \begin{enumerate}[label=\textup{(\alph{*})}]
  \item The function \(\trpi\) is defined and constant on \(\tilde{H}_{\rmax{+}} \gamma \cap \T\), and on all \(\G\)\nb-conjugacy classes intersecting this set.
  \item The function \(\trpi\) is constant on \(\UC[x]^{(\rmax + d_\T (x))}\gamma \), for any \(x \in \bt{\F}\).
  \end{enumerate}
\end{theorem}

\begin{proof}
  For every root \(\alpha \in \Phi\bigl(\Galg(\tilde{\F}),\Talg(\tilde{\F})\bigr)\) and every \(g \in \tilde{H}_{\rmax{+}} \gamma \cap \T\) we have \(\tilde{v}\bigl(\alpha (g)\bigr) = \tilde{v}\bigl(\alpha(\gamma)\bigr)\) because \(g \gamma^{-1}\) is compact.  Together with~\eqref{eq:PhiPg}, this implies \(\Pa_g = \Pa_\gamma\), so that Theorem~\ref{thm:trJacquet} applies to all \(g \in \tilde{H}_{\rmax{+}} \gamma \cap \T\) and tells us that \(\trpi(g) = \tr_{\repr_{\Ru[\Pa_\gamma]}} (g)\).  Theorem~\ref{thm:localconst} and Proposition~\ref{prop:levele} show that \(\tr_{\repr_{\Ru[\Pa_\gamma]}}\) is constant on \(\tilde{H}_{\rmax{+}} \gamma \cap \T\), so the same goes for \(\trpi\).  This proves~(a), from which~(b) follows upon applying Lemma~\ref{lem:conjugatedTx}.
\end{proof}

This theorem is similar to \cite{Adler-Korman:Local_character}*{Corollary 12.11}, which was proved only for complex representations and ``tame'' elements~\(\gamma\).  Our neighbourhoods of constancy are usually smaller than those in~\cite{Adler-Korman:Local_character}, because Theorem~\ref{thm:localconst}.(a) is not optimal.  The results of Adler and Korman suggest that Theorem~\ref{thm:localconst}.(c) could be valid whenever the maximal torus~\(\T\) splits over a tamely ramified extension of~\(\F\).  Possibly this has something to do with Rousseau's result \eqref{eq:Galoisinv}.

\section{A bound for the dimension of \texorpdfstring{$V^\CG$}{fixed-point subspaces}}
\label{sec:bound}

In this section, we will use the resolutions of~\cite{Meyer-Solleveld:Resolutions} to estimate the dimension of~\(V^{\UC[x]^{(e)}}\) for an admissible representation \((\repr, V)\) of~\(\G\) in good characteristic.  We abbreviate \(\CG_e\defeq\UC[x]^{(e)}\).

First we estimate the growth of some related double coset spaces in order to show that our later estimates are optimal, at least for \(\GL_n\).

Since every irreducible smooth representation is a subquotient of a parabolically induced one, the essential case is \(V = \Ind_\Pa^\G (W)\), where~\(\Pa\) is a parabolic subgroup of~\(\G\) and \((\rho ,W)\) is a supercuspidal representation of \(\Pa / \Ru\).  There is a natural isomorphism
\begin{equation}
  \label{eq:VKe}
  V^{\CG_e} \cong \bigoplus_{\Pa g \CG_e} W^{\Pa \cap g \CG_e g^{-1}},
\end{equation}
where the sum runs over all double \((\Pa,\CG_e)\)-cosets.  The space \(\Pa \backslash \G / \CG_e\) is finite because \(\Pa \backslash \G\) is a complete algebraic variety (and hence compact in the \(p\)\nb-adic topology) and~\(\CG_e\) is open.  We will discuss how \(\abs{\Pa \backslash \G / \CG_e}\) grows as~\(e\) increases, under some simplifications.  If~\(\Pa\) is a Borel subgroup and~\(\rho\) is a character, then \(\abs{\Pa \backslash \G / \CG_e}\) and \(\dim V^{\CG_e}\) have equivalent growth rates.

Suppose that~\(\G\) is split.  Let~\(\ST\) be a split maximal torus of~\(\G\) and let~\(\Pa_D\) be a standard parabolic subgroup of~\(\G\).  The dimension of \(\Pa_D \backslash \G\) is
\[
\dim (\Pa_D \backslash \G)
= \dim_{\F} \bigl( \Lie (\G) / \Lie (\Pa_D ) \bigr)
= \sum_{\alpha \in \Phi^- \backslash \Phi_D^-} \dim_\F \Lie (\Un_\alpha)
= \abs{\Phi^-} - \abs{\Phi_D^-}.
\]
Let \(x \in \Apa\).  By construction, the groups~\(\CG_e\) decrease equally fast in every direction; if~\(\CG_e\) corresponds to a lattice~\(L^{(e)}\) in \(\Lie (\G)\), then~\(\CG_{e+1}\) corresponds to~\(\maxid L^{(e)}\), where~\(\maxid\) is the maximal ideal in the maximal compact subring of~\(\F\).  Hence a double coset \(\Pa_D g \CG_e\) contains approximately~\(q^{\dim (\Pa_D \backslash \G)}\) double \(\bigl(\Pa_D,\CG_{e+1} \bigr)\)-cosets.  Therefore, \(\abs{\Pa_D \backslash \G / \CG_e}\) grows, in first approximation, like \(q^{e \dim (\Pa_D \backslash \G)}\).

Now we focus on the easier example \(\G=\GL_n\) and let \(\Pa\) and~\(\ST\) be the standard Borel subgroup and the standard maximal torus in \(\GL_n (\F)\).  The irreducible representations of \(\ST = \Pa / \Ru\) are characters.  Let \((\rho,\C)\) be such a character and let~\(V\) be the parabolically induced representation of~\(\G\).  Since any character is trivial on \(\CG_e\cap\ST\) for large enough~\(e\), \(\C^{\Pa \cap g \CG_e g^{-1}}\cong\C\) for large enough~\(e\), so that \(\dim (V^{\CG_e}) = \abs{\Pa\backslash \G/\CG_e}\) for large~\(e\).  These numbers are routine to compute:
\begin{equation}
  \label{eq:growthGLn}
  \abs{\Pa \backslash \G / \CG_e} \approx e^{n-1} q^{en(n-1)/2}
\end{equation}
in the sense that the quotient of both sides tends towards a constant as \(e \to \infty\).

For complex representations, we may use the growth rate of \(\dim V^{\CG_e}\) to estimate the growth of the character.  It will, however, turn out that these estimates are far from optimal.  The idea is simple enough: if \(\trpi\) is constant on~\(\CG_e\gamma\), then
\[
\trpi(\gamma)
= \frac{1}{\bigl|\CG_e\gamma\bigr|}
\int_{\CG_e\gamma} \trpi(\gamma) \,\diff\mu(\gamma)
= \tr\bigl(\pi(\ide{\CG_e}\gamma)\bigr).
\]
Equip the finite-dimensional vector space~\(V^{\CG_0}\) with some norm.  Since the range of \(\ide{\CG_e}\gamma\) is contained in \(V^{\CG_e} \subseteq V^{\CG_0}\) and the largest eigenvalue of \(\ide{\CG_e}\gamma\) is controlled by the operator norm \(\norm{\ide{\CG_0}\gamma\ide{\CG_0}}_\infty\), we get the estimate
\begin{equation}
  \label{eq:trace_by_dimension_estimate}
  \abs{\trpi(\gamma)}
  \\\le \norm{\ide{\CG_0}\gamma\ide{\CG_0}}_\infty\cdot \dim V^{\CG_e}.
\end{equation}
Since the function \(\gamma\mapsto \ide{\CG_0}\gamma\ide{\CG_0}\) is locally constant, the \emph{local} growth of the right hand side is equivalent to that of \(\dim V^{\CG_e}\).  This depends on~\(\gamma\) via~\(e\).  For~\(x\) sufficiently close to the set of singular elements (namely, for \(\sd(\gamma)>e+d(\gamma)\)) we may take \(e=\sd(\gamma)\) by Theorem~\ref{thm:localconst}.

Unfortunately, a direct computation for \(\GL_n\) shows that
\[
\sum_{e=0}^\infty \dim V^{\CG_e} \cdot \mu\{g\in \CG_0 : \sd(g)=e\}
\]
diverges, already for \(\GL_2\).  Hence the estimate~\eqref{eq:trace_by_dimension_estimate} does not imply the local integrability of~\(\trpi\).  The authors have not been able to detect the additional cancellation in our trace formula that makes the character locally integrable.

Instead, we estimate of the growth of \(\dim V^{\CG_e}\).  For convenience, we assume that \(x = o\) is the origin of the apartment~\(\Apa\) and that \(e \in \Z_{\ge 0}\).

Theorem~\ref{thm:resolution} assigns to every convex subcomplex~\(\Sigma\) of \(\bt{\F}\) a subspace of~\(V\), namely the image \(\sum_{x\in\Sigma^\circ} V^{\UC[x]^{(e)}}\) of \(\partial_0\colon C_0 (\Sigma ,V) \to V\).  This space admits an important alternative description if~\(\Sigma\) is finite.

\begin{theorem}[\cite{Meyer-Solleveld:Resolutions}*{Theorem 2.12}]
  \label{thm:supproj}
  The elements
  \[
  u_\Sigma^{(e)} \defeq
  \sum_{\sigma \in \Sigma} (-1)^{\dim\sigma} \ide{\UC[\sigma]^{(e)}} \in \Hec
  \]
  are idempotent and
  \begin{align*}
    u_{\Sigma}^{(e)} \Hec
    &= \sum_{x \in \Sigma^\circ} \ide{\UC[x]^{(e)}} \Hec,\\
    \bigl(1 - u_\Sigma^{(e)} \bigr) \Hec &=
    \bigcap_{x \in \Sigma^\circ} \bigl(1 - \ide{\UC[x]^{(e)}} \bigr)  \Hec.
  \end{align*}
  In particular,
  \[
  \image(\partial_0\colon C_0 (\Sigma ,V) \to V)
  = \sum_{x \in \Sigma^\circ} V^{\UC[x]^{(e)}} = u_\Sigma^{(e)} V.
  \]
\end{theorem}

It is shown in~\cite{Meyer-Solleveld:Resolutions} that there is a convex subcomplex~\(\Sigma_0\) such that \(\ide{\UC[o]^{(r)}} u_\Sigma^{(e)} = \ide{\UC[o]^{(r)}} u_{\Sigma_0}^{(e)}\) for all convex subcomplexes~\(\Sigma\) with \(\Sigma\supseteq\Sigma_0\).  The following lemma describes~\(\Sigma_0\) explicitly.  To state it, we need some notation.  For \(\alpha \in \Phi\) we define
\begin{align*}
  \Apa[\ST,r]^{\alpha{+}} &\defeq \{x\in\Apa : \inp{x}{\alpha} > r \},\\
  \Apa[\ST,r]^{\alpha{0}} &\defeq \{x\in\Apa : \inp{x}{\alpha} \in [-r,r] \},\\
  \Apa[\ST,r]^{\alpha{-}} &\defeq \{x\in\Apa : \inp{x}{\alpha} < -r \},
\end{align*}
and for any map \(\epsilon\colon \Phi\to \{+,0,-\}\) we write
\[
\Apa[\ST,r]^\epsilon \defeq
\bigcap_{\alpha \in \Phi} \Apa[\ST,r]^{\alpha,\epsilon(\alpha)}.
\]
Most of the sets \(\Apa[\ST,r]^\epsilon\) are empty, some are compact, and the others are unbounded.  The non-empty \(\Apa[\ST,r]^\epsilon\) partition~\(\Apa[\ST]\).  Let \(\Apa[\ST,r]^b\) be the union of the bounded \(\Apa[\ST,r]^\epsilon\); this is a polysimplicial subcomplex of~\(\Apa\) which is star-shaped around~\(o\).  The subcomplex \(B_r \defeq \Pa_o \cdot \Apa[\ST,r]^b\) of \(\bt{\F}\) is obviously stable under the action of all the groups \(\UC[o]^{(s)}\) for \(s \in \R_{\ge 0}\).  We may think of~\(B_r\) as a combinatorial approximation to a ball of radius~\(r\) around~\(o\).

\begin{lemma}
  \label{lem:cancel}
  Let \(r \in \Z_{\ge e}\) and let \(\Sigma \subseteq \bt{\F}\) be any finite convex subcomplex that contains~\(B_{r-e}\).  Then
  \[
  \ide{\UC[o]^{(r)}} u_\Sigma^{(e)}
  = \ide{\UC[o]^{(r)}} u_{B_{r-e}}^{(e)}
  = \sum_{\sigma \in B_{r-e}} (-1)^{\deg \sigma}  \ide{\UC[o]^{(r)}} \ide{\UC[\sigma]^{(e)}}.
  \]
\end{lemma}

\begin{proof}
  Fix \(\epsilon\colon \Phi\to \{ +,0,- \}\) such that \(\Apa[\ST, r-e]^\epsilon\) is \emph{unbounded}.  First we establish \(\ide{\UC[o]^{(r)}} \ide{\UC[F]^{(e)}} = \ide{\UC[o]^{(r)}} \ide{\UC[F']^{(e)}}\) for certain facets \(F ,F' \subseteq \Apa[\ST,r-e]^\epsilon\).  The coroots \(\alpha^\vee \in \Phi^\vee\) with \(\epsilon (\alpha) = 0\) span a proper subspace \(\Apa[\ST,\perp]^\epsilon \subsetneq \Apa\).  We may pick a non-zero vector \(\delta^\epsilon \in \Apa\) such that
  \begin{enumerate}
  \item \(\delta^\epsilon\) is orthogonal to \(\Apa[\ST,\perp]^\epsilon\),
  \item \(\Apa[\ST,r-e]^\epsilon + \R_{\ge 0} \delta^\epsilon \subseteq \Apa[\ST,r-e]^\epsilon\),
  \item \(\delta^\epsilon\) lies in the span of an irreducible root subsystem~\(\Psi^\vee\) of~\(\Phi^\vee\) (here we decompose~\(\Phi^\vee\) as a direct sum of irreducible root systems).
  \end{enumerate}
  For every facet \(F \subseteq \Apa[\ST,r-e]^\epsilon\) let \(M(F) \subseteq \Apa[\ST,r-e]^\epsilon\) be the unique facet such that for all \(x \in F\) there exists \(\lambda > 0\) with \(x + \lambda \cdot \delta^\epsilon \subseteq M(F)\).  We claim that
  \begin{equation}
    \label{eq:uuuu}
    \ide{\UC[o]^{(r)}} \ide{\UC[F]^{(e)}}
    = \ide{\UC[o]^{(r)}} \ide{\UC[M(F)]^{(e)}}
    \qquad\text{for \(F \subseteq \Apa[\ST,r-e]^\epsilon\).}
  \end{equation}
  In view of the unique decomposition property (Proposition~\ref{prop:udp}) this is equivalent to
  \[
  \bigl( \UC[o]^{(r)} \cup \UC[F]^{(e)} \bigr) \cap \Un_\alpha
  = \bigl( \UC[o]^{(r)} \cup \UC[M(F)]^{(e)} \bigr) \cap \Un_\alpha
  \qquad \text{for all \(\alpha \in \Phr\).}
  \]
  By definition, \(\UC[o]^{(r)} \cap \Un_\alpha = \Un_{\alpha, r{+}}\) and \(\UC[F]^{(e)} \cap \Un_\alpha = \Un_{\alpha ,-\alpha (x)+e{+}}\) for \(x \in F\).  If \(\epsilon (\alpha) = {-}\), then \(-\alpha +e > r\) on \(F \cup M(F)\), so that
  \[
  \Un_{\alpha ,r{+}} \supseteq \Un_\alpha \cap \bigl( \UC[o]^{(r)} \cup \UC[M(F)]^{(e)} \bigr).
  \]
  If \(\epsilon (\alpha) \neq {-}\), then \(\sup_{x \in F} {-}\alpha (x) \le \sup_{x \in M(F)} {-}\alpha (x)\), which combined with \(\UC[F]^{(e)} \subseteq \UC[M(F)]^{(e)}\) yields \(\UC[F]^{(e)} \cap \Un_\alpha = \UC[M(F)]^{(e)} \cap \Un_\alpha\).  This finishes the proof of~\eqref{eq:uuuu}.

  Now we use~\eqref{eq:uuuu} to establish some cancellation.  Every facet~\(F\) in~\(\Apa\) can be written uniquely as \(F = F_\Psi \times F_\perp\), where \(F_\Psi\) and~\(F_\perp\) are facets in \(\R \Psi^\vee\) and \(\Psi^\perp \subseteq \Apa\), respectively.  Consider a facet \(F \subseteq \Apa[\ST,r-e]^\epsilon\) such that \(M^{-1}(F)\) is not empty.  Then \(M(F) = F\), and \(M^{-1}(F)\) consists of facets of~\(\cl{F}\).  Property~(3) above shows that \(F'_\perp = F_\perp\) for any \(F' \in M^{-1}(F)\).  Hence
  \[
  \bigcup_{F' \in M^{-1}(F)} F' = \tau \times F_\perp ,
  \]
  where \(\tau \subseteq \R \Psi^\vee\) consists of the facets of~\(\cl{F_\Psi}\) that contain points of the form \(x + \lambda \delta^\epsilon\) with \(x \in F\) and \(\lambda \ge 0\).  In particular, \(\tau\) is diffeomorphic to
  \[
  (-1,1] \delta^\epsilon + \{ x \in F : \inp{x}{\delta^\epsilon} = c \}
  \]
  for some \(c \in \R\), so that the Euler characteristic of~\(\tau\) is zero.  Therefore,
  \begin{multline}
    \label{eq:Eulertau}
    \sum_{F' \in M^{-1}(F)} (-1)^{\dim F'} =
    \sum_{F' \in M^{-1}(F)} (-1)^{\dim F'_\Psi} (-1)^{\dim F_\perp} = \\
    \sum_{\text{\(\tau'\) facet in \(\tau\)}} (-1)^{\dim \tau'} (-1)^{\dim F_\perp}= 0 ,
  \end{multline}
  which together with~\eqref{eq:uuuu} yields
  \begin{equation}
    \label{eq:cancel}
    \sum_{F' \in M^{-1}(F)} (-1)^{\dim F'} \ide{\UC[o]^{(r)}} \ide{\UC[F']^{(e)}} \; = \; 0 \;
    \in \Hec.
  \end{equation}
  Suppose that~\(\Apa\) is any apartment of \(\bt{\F}\) that contains~\(o\) and at least one facet \(F' \in M^{-1}(F)\).  As~\(\delta^\epsilon\) points away from~\(o\), the apartment~\(\Apa\) contains points of~\(F\), so that \(\cl{F} \subseteq \Apa\).  This enables us to extend the map~\(M\) to all facets of \(\bt{\F}\).  Recall that any Weyl chamber \(\Apa^+ \subseteq \Apa\) is a fundamental domain for the action of~\(\Pa_o\) on \(\bt{\F}\).  On~\(\Apa^+\) we define~\(M\) according to the above recipe and by \(M(F) \defeq F\) if \(F \subseteq \Apa[\ST,r-e]^b \cap \Apa^+\).  The properties (1)--(3) of~\(\delta^\epsilon\) ensure that \(M(F)\) and~\(F\) have the same isotropy group in~\(\Pa_o\), so we can extend~\(M\) \(\Pa_o\)\nb-equivariantly to \(\bt{\F}\).

  Since~\(\Sigma\) contains~\(o\) and is a convex subcomplex of \(\bt{\F}\), its collection of facets is stable under~\(M\).  By definition
  \begin{multline*}
    \ide{\UC[o]^{(r)}} u_\Sigma^{(e)}
    = \ide{\UC[o]^{(r)}} \sum_{\sigma\in\Sigma} (-1)^{\deg\sigma}\ide{\UC[\sigma]^{(e)}}
    \\= \ide{\UC[o]^{(r)}} \sum_{\text{\(F\)~facet of~\(\Sigma\)}} \
    \sum_{F' \in M^{-1}(F)} (-1)^{\dim F'} \ide{\UC[F']^{(e)}}.
  \end{multline*}
  Now~\eqref{eq:cancel} (which only holds for facets of unbounded \(\Apa[\ST, r-e]^\epsilon\)) shows that the facets of \(\Sigma \backslash B_{r-e}\) do not contribute to this sum.  As~\(M\) is the identity on facets of~\(B_{r-e}\), we remain with \(\ide{\UC[o]^{(r)}} u_\Sigma^{(e)} = \ide{\UC[o]^{(r)}} u_{B_{r-e}}^{(e)}\).
\end{proof}

\begin{remark}
  \label{rem:cancel_lemma}
  Lemma~\ref{lem:cancel} provides a direct proof of the special case of \cite{Meyer-Solleveld:Resolutions}*{Proposition 3.6} where the consistent system of idempotents is \(\ide{\UC[x]^{(e)}}\); this proof does not use the fact that the Hecke algebra is Noetherean.
\end{remark}

We turn to the space of invariants~\(V^{\UC[o]^{(r)}}\).  Since it has finite dimension, it is contained in the range of~\(u_\Sigma^{(e)}\) for some finite convex subcomplex \(\Sigma \subseteq \bt{\F}\).  We may as well assume that~\(\Sigma\) contains~\(B_{r-e}\), so that Lemma~\ref{lem:cancel} yields
\[
V^{\UC[o]^{(r)}}
= \ide{\UC[o]^{(r)}} u_\Sigma^{(e)} V
= \Bigl(\sum_{\sigma \in B_{r-e}} (-1)^{\deg\sigma}
\ide{\UC[o]^{(r)}} \ide{\UC[\sigma]^{(e)}} \Bigr) V.
\]
The right hand side is contained in \(\sum_{x \in B_{r-e}^\circ} \ide{\UC[o]^{(r)}} \ide{\UC[x]^{(e)}} V\) by Theorem~\ref{thm:UOmegae}.\ref{UOmegae_4}.  It is the space of \(\UC[o]^{(r)}\)\nb-invariants in \(\sum_{x \in B_{r-e}^\circ} \ide{\UC[x]^{(e)}} V\) because \(\sum_{x \in B_{r-e}^\circ} \ide{\UC[x]^{(e)}} V\) is \(\Pa_o\)\nb-invariant.  Let \(\Pa_o\supseteq \ide{\UC[o]^{(r)}}\) act on \(\bigoplus_{x \in B_{r-e}^\circ} \ide{\UC[x]^{(e)}} V\) by \(g \cdot (x,v) = (g \cdot x, \repr (g) v)\).  Then \(\sum_{x \in B_{r-e}^\circ} \ide{\UC[o]^{(r)}} \ide{\UC[x]^{(e)}} V\) is a quotient of \(\bigoplus_{x \in B_{r-e}^\circ} \ide{\UC[x]^{(e)}} V\).  The addition map
\[
\biggl( \bigoplus_{x \in B_{r-e}^\circ} \ide{\UC[x]^{(e)}} V \biggr)^{\UC[o]^{(r)}} \to
\biggl( \sum_{x \in B_{r-e}^\circ} \ide{\UC[x]^{(e)}} V \biggr)^{\UC[o]^{(r)}}
\]
is surjective because~\(\UC[o]^{(r)}\) is compact and we are working in good characteristic.  Since there are only finitely many \(\G\)\nb-orbits of vertices in \(\bt{\F}\),
\begin{equation}
  \label{eq:MV}
  m_V \defeq \max_{x \in \bt{\F}} \dim V^{\UC[x]^{(e)}}
\end{equation}
exists.  The dimension of \(\bigl( \bigoplus_{x \in B_{r-e}^\circ} \ide{\UC[x]^{(e)}} V \bigr)^{\UC[o]^{(r)}}\) is at most \(m_V \abs{B_{r-e}^\circ / \UC[o]^{(r)}}\).

It remains to estimate the number of \(\UC[o]^{(r)}\)-orbits of vertices in~\(B_{r-e}\).  For \(\alpha \in \Phi\) let~\(d_\alpha\) be the dimension of \(\Lie(\Ualg_\alpha / \Ualg_{2 \alpha})\) and let~\(d_0\) be the dimension of \(\Lie(\Calg(\Salg))\).  Recall that \(q = \abs{\integ / \mathfrak \Pa}\) and that~\(n_\alpha^{-1} \Z\) is the set of jumps of the filtration of~\(\Un_\alpha\).

\begin{lemma}
  \label{lem:orbits}
  The number of \(\UC[o]^{(r)}\)-orbits on~\(B_{r-e}^\circ\) is of order \(\Oest(r^{\dim \Apa} Q^r)\), where
  \[
  Q\defeq \exp \biggl( \log (q) \sum_{\alpha\in\Phr}
  \frac{d_\alpha n_\alpha}{2} + \frac{d_{2 \alpha} n_{2 \alpha}}{4} \biggr).
  \]
\end{lemma}

\begin{proof}
  Recall from~\eqref{eq:POmega} and Proposition~\ref{prop:decompose}.\ref{prop:decomposec} that
  \[
  \Pa_o = \UC[o] \NC[o] = \UC[o]^+ \UC[o]^- (\Pa_o \cap \No(\ST)) ,
  \]
  for any positive root system \(\Phi^+\) of \(\Phi\).  Hence every facet of \(B_{r-e} = \Pa_o \cdot \Apa[\ST,r-e]^b\) is of the form~\(u \cdot F\) with \(u \in \UC[o]^+ \UC[o]^-\) and a facet~\(F\) of~\(\Apa\).  Fix~\(F\) and choose a positive root system~\(\Phi^+\) such that \(\alpha (F) \ge 0\) for all \(\alpha \in \Phi^+\).  Then \(\UC[o]^- \subseteq \UC[F]^-\) fixes~\(F\) pointwise, so that we only need \(u \in \UC[o]^+\).  By Propositions \ref{prop:decompose}.\ref{prop:decomposeb} and~\ref{prop:udp} the product maps
  \[
  \prod_{\alpha\in\Phr\cap\Phi^+} \Un_{\alpha,0} \to \UC[o]^+,\qquad
  \prod_{\alpha\in\Phr\cap\Phi^+} (\Un_\alpha \cap \UC[o]^{(r)}) \to \Un^+ \cap \UC[o]^{(r)}
  \]
  are diffeomorphisms.  Together with the conventions~\eqref{eq:Ualphar} we get
  \begin{equation}
    \label{eq:Uo+}
    \begin{aligned}[t]
      [\UC[o]^+ : \UC[o]^+ \cap \UC[o]^{(r)}]
      &= \prod_{\alpha\in\Phr\cap\Phi^+ \!\!\!} [\Un_{\alpha,0} : \Un_\alpha \cap \UC[o]^{(r)}]
      \\&= \prod_{\alpha\in\Phr\cap\Phi^+ \!\!\!} [\Un_{\alpha,0} \Un_{2 \alpha ,0}: \Un_{\alpha ,r{+}} \Un_{2 \alpha, 2r{+}}]
      \\&= \prod_{\alpha\in\Phr\cap\Phi^+ \!\!\!} [\Un_{\alpha,0} / \Un_{2\alpha,0} : \Un_{\alpha ,r{+}} / \Un_{2 \alpha, 2r{+}}] \cdot [\Un_{2\alpha,0}: \Un_{2\alpha, 2r{+}}].
    \end{aligned}
  \end{equation}
  Since we are dealing with unipotent pro-\(p\)-groups, these indices can be read off from the Lie algebras.  For \(\alpha \in \Phi\) and \(s\in n_\alpha^{-1} \Z\), the construction from \eqref{eq:Usplit} and~\eqref{eq:Ualphar} shows that \(\Un_{\alpha,s} \supsetneq \Un_{\alpha,s{+}}\) corresponds to multiplying a lattice in \(\Lie(\Ualg_\alpha)\) with the maximal ideal~\(\mathfrak \Pa\) of~\(\integ\), see also \cite{Tits:Reductive}*{3.5.4}.  Hence
  \begin{align*}
    [\Un_{\alpha ,s} / \Un_{2 \alpha ,2s} : \Un_{\alpha ,s{+}} / \Un_{2\alpha ,2s{+}} ] &= q^{d_\alpha}, \\
    [\Un_{\alpha ,0} / \Un_{2 \alpha ,0} : \Un_{\alpha ,r{+}} / \Un_{2\alpha ,2r{+}} ] &= q^{d_\alpha \lceil n_\alpha r{+} \rceil},
  \end{align*}
  where \(\lceil y{+} \rceil\) denotes the smallest integer larger than \(y+ \in \tilde{\R}\).  Similarly
  \[
  [\Un_{2 \alpha ,0} : \Un_{2 \alpha ,r{+}}] = q^{d_{2\alpha} \lceil n_{2\alpha} r/2 {+} \rceil},
  \]
  from which we conclude that
  \begin{multline}
    \label{eq:prodq}
    [\UC[o]^+ : \Un^+ \cap \UC[o]^{(r)}]
    = \prod_{\alpha \in \Phr \cap \Phi^+} q^{d_\alpha \lceil n_\alpha r{+} \rceil}
  q^{d_{2 \alpha} \lceil n_{2\alpha} r/2 + \rceil} \\
    \le \prod_{\alpha\in\Phr} q^{d_\alpha ( n_\alpha r + 1)/2} q^{d_{2 \alpha} (n_{2\alpha} r + 2) /4}.
  \end{multline}
  This number is an upper bound for the number of \(\UC[o]^{(r)}\)\nb-orbits in \(\UC[o] \cdot F\).  Since it does not depend on~\(F\), we only need to multiply it with the number of facets of \(\Apa[\ST,r-e]^b\).  While this number is not easily expressible in a formula, it clearly grows like~\(r^{\dim \Apa}\).
\end{proof}

\begin{theorem}
  \label{thm:estdim}
  Let \((\repr ,V)\) be an admissible \(\G\)\nb-representation of level \(e \in \Z_{\ge 0}\) in good characteristic.  Let \(r\in\R_{\ge e}\) and define \(Q\) and~\(m_V\) as in Lemma~\textup{\ref{lem:orbits}} and~\eqref{eq:MV}.  Then
  \begin{align*}
    \dim V^{\UC[o]^{(r)}} &= \Oest(m_V r^{\dim \Apa} Q^r)\\
    \mu \bigl( \UC[o]^{(r)} \bigr) \dim V^{\UC[o]^{(r)}}
    &= \Oest (m_V r^{\dim \Apa} q^{-r d_0} Q^{-r})
  \end{align*}
  with constants independent of~\(V\) and~\(r\).
\end{theorem}

\begin{proof}
  The first estimate follows from Lemma~\ref{lem:orbits} and the arguments above.  Proposition~\ref{prop:udp} yields
  \[
  [\UC[o]^{(s)} : \UC[o]^{(r+s)}]
  = [H_{s{+}} : H_{r+s{+}}] \prod_{\alpha \in \Phr}
  [\Un_{\alpha ,s{+}} \Un_{2 \alpha ,s{+}}: \Un_{\alpha ,r+s{+}} \Un_{2 \alpha, r+s{+}}]
  \]
  for all \(s \in \Z_{\ge 0}\).  A calculation like the one in \eqref{eq:Uo+} and~\eqref{eq:prodq} shows that this index is at least
  \[
  q^{r d_0} \prod_{\alpha \in \Phr} q^{r n_\alpha d_\alpha} q^{r n_{2 \alpha} d_{2 \alpha} /2}.
  \]
  (We cannot be exact because we do not know at which points the filtration of~\(H\) jumps.)  This yields the second estimate.
\end{proof}

These estimates are sharp in some examples: \eqref{eq:growthGLn} shows that (a) and~(c) cannot be improved for \(\GL_n\).  Here all \(n_\alpha\) and~\(d_\alpha\) are~\(1\), \(\Phi\) is reduced, and there are \(n(n-1)/2\) positive roots, so that \(Q=q^{n(n-1)/2}\).

\section{Conclusion}
\label{sec:conclusion}

Let~\(\G\) be a reductive \(p\)\nb-adic group and let \((\rho,V)\) be an admissible representation of~\(\G\) on a vector space~\(V\) of characteristic not equal to~\(p\).  We have seen that the character of~\((\rho,V)\) is a locally constant function on the set of regular semi-simple elements, and we have described explicit open subsets on which it is constant.  Furthermore, we have estimated the growth of the dimensions of the fixed-point subspaces \(V^{\UC[x]^{(e)}}\) for \(e\to\infty\).  Both results are based on the main result of~\cite{Meyer-Solleveld:Resolutions} about the acyclicity of certain coefficient systems on the affine Bruhat--Tits building.

It is still unclear whether Harish-Chandra's theorem about the local integrability of the character function for complex representations can be established using these resolutions.  This may depend on a better understanding of the character formulas.  While the resolution in~\cite{Meyer-Solleveld:Resolutions} does provide an explicit formula for the character, more work is required to understand and simplify this formula.

\end{document}